\documentclass[11pt]{amsart} 
\usepackage{amsmath, amssymb, amsfonts, amsbsy, amsthm, latexsym, stmaryrd, mathtools, fullpage, enumerate, tikz, standalone}
\usepackage{geometry}
\usepackage{graphicx,float}
\usepackage[affil-it]{authblk}
\usepackage[mathscr]{eucal}
\usepackage{wrapfig}
\usepackage{todonotes}
\usepackage{tikz-cd}
\usepackage[foot]{amsaddr}

\usepackage[mathscr]{eucal}
\usepackage{url}
\usepackage{shuffle}

\usepackage{amsmath}
\def\sh{\shuffle}

\newtheorem{thm}{Theorem}[section]
\newtheorem*{thm*}{Theorem}

\newtheorem{lem}[thm]{Lemma}
\newtheorem{prop}[thm]{Proposition}
\newtheorem{cor}[thm]{Corollary}
\theoremstyle{definition}
\newtheorem{remark}[thm]{Remark}
\newtheorem{defn}[thm]{Definition}
\newtheorem{example}[thm]{Example}

\newtheorem{conj}[thm]{Conjecture}

\numberwithin{equation}{section}

\def\C{{\mathbb{C}}}
\def\Z{{\mathbb{Z}}}
\def\Q{{\mathbb{Q}}}
\def\P{{\mathbb{P}}}
\DeclareMathOperator{\D}{\textbf{D}}

\def\degb{{\deg_\mathcal{B}}}
\def\zm{{\zeta^{\mathfrak{m}}}}
\def\za{{\zeta^{\mathfrak{a}}}}
\def\z{{\zeta}}

\def\im{{\text{I}^{\mathfrak{m}}}}
\def\il{{\text{I}^{\mathfrak{l}}}}
\def\zb{{\zeta^{\mathfrak{b}}}}
\def\zbl{{\zeta^{\mathfrak{bl}}}}
\def\ia{{\text{I}^{\mathfrak{a}}}}
\def\ib{{\text{I}^\mathfrak{b}}}
\def\ibl{{\text{I}^\mathfrak{bl}}}
\def\iform{{\text{I}^\mathfrak{f}}}
\def\bg{{\mathfrak{bg}}}
\def\rbg{{\mathfrak{rbg}}}
\def\gmot{{\mathfrak{g}^\mathfrak{m}}}

\DeclareMathOperator{\gr}{gr}
\def\qpoly{{\Q\langle e_0,e_1\rangle}}

\def\dmr{{\mathfrak{dmr}_0}}
\def\dg{{\mathfrak{dg}}}
\def\ls{{\mathfrak{ls}}}

\DeclareMathOperator{\Sh}{Sh}

\def\Gmot{{\text{G}_{\mathcal{M}\mathcal{T}(\mathbb{Z})}}}
\def\Umot{{\text{U}_{\mathcal{M}\mathcal{T}(\mathbb{Z})}}}
\DeclareMathOperator{\Lie}{Lie}
\DeclareMathOperator{\id}{id}

\DeclareMathOperator{\mod2}{(mod\ 2)}
\DeclareMathOperator{\Cyc}{C}
\DeclareMathOperator{\Sym}{Sym}
\DeclareMathOperator{\sgn}{sgn}

\DeclareMathOperator{\ic}{I}

\newcommand{\partialx}[1]{\frac{\partial}{\partial x_{#1}}}

\def\proj{{\P^1\setminus\{0,1,\infty\}}}

\title{Motivic Multiple Zeta Values and the Block Filtration}
\author{Adam Keilthy}
\address{Max Planck Institute for Mathematics, Bonn}
\email{keilthy@mpim-bonn.mpg.de}
\thanks{This work would not have been possible without the combined support of the Clarendon Fund, University of Oxford, and the Max Planck Institute for Mathematics. The author would also particularly like to thank the referee for their very helpful advice, and suggestion to consider Proposition \ref{charlalt}.}

\begin{document}
\maketitle

\begin{abstract}
We extend the block filtration, defined by Brown based on the work of Charlton, to all motivic multiple zeta values, and study relations compatible with this filtration. We construct a Lie algebra describing relations among motivic multiple zeta values modulo terms of lower block degree, proving Charlton's cyclic insertion conjecture in this structure, and showing the existence of a `block shuffle' relation, a dihedral symmetry, and differential relation.
\end{abstract}
\section{Introduction}
The study of multiple zeta values (MZVs)
$$\z(n_1,\ldots,n_r) := \sum_{1\leq k_1<k_2<\cdots<k_r}\frac{1}{k_1^{n_1}\ldots k_r^{n_r}}$$
goes back to Euler, and finds applications in many areas of mathematics, including the study of associators \cite{drinfeld}, knot theory \cite{barnatan}, and quantum field theory \cite{bkconj}. Of particular interest is describing all algebraic relations among multiple zeta values. When considering single zeta values, we have the standard conjecture:
\begin{conj}
There are no algebraic relations among $\pi$ and the odd zeta values, i.e.
$$\{\pi,\z(3),\z(5),\ldots\}$$
is algebraically independent.
\end{conj}

In contrast, multiple zeta values satisfy many algebraic relations: the double shuffle relations \cite{racinet}, the associator relations \cite{drinfeld}, the derivation relations \cite{dualder}, and more.
\begin{example} By splitting up the domain of integration (see Example \ref{zeta2int}) or the domain of summation in the product $\z(2)^2$, we obtain:
\begin{align*}
\z(2)^2 &= 2\z(2,2)+ 4\z(1,3)\\
&= 2\z(2,2)+\z(4)
\end{align*}
\end{example}
While some of these, such as the double shuffle relations, are conjectured to describe all relations among MZVs, no complete description is known. One approach to this problem is to lift MZVs to an algebra of formal objects called motivic MZVs via their representation as iterated integrals on $\proj$.

Relations among motivic MZVs are then encoded by the motivic Galois group of $\mathcal{M}\mathcal{T}(\mathbb{Z})$, the category of mixed Tate motives over Spec($\Z$) \cite{brownmixedtate}. From the work of Deligne and Goncharov \cite{delmixedtate}, the algebraic structure of this affine group scheme is well understood. It decomposes as a semidirect product of the multiplicative group and a pro-unipotent group, whose graded Lie algebra is non-canonically isomorphic to a free Lie algebra with generators in odd weight
$$\gmot\cong\text{Lie}[\sigma_3,\sigma_5,\ldots]$$
 called the motivic Lie algebra.
There exists a non-canonical injection $i:\{\sigma_{2k+1}\}_{k\geq 1}\to\qpoly$, allowing us to consider elements of $\gmot$ as noncommutative polynomials. This Lie algebra has ties to the Grothendieck-Teichmuller group \cite{racinet}, associators \cite{drinfeld}, and multiple zeta values \cite{brownmixedtate}, and injects into Racinet's double shuffle Lie algebra $\dmr$ \cite{racinet}. Describing its image in $\qpoly$ explicitly is of particular interest, as coefficients of elements of the image describe relations among motivic MZVs modulo products.

One way to pursue this is to consider an associated graded Lie algebra. We can see that $\gmot$ inherits two filtrations arising from filtrations on motivic multiple zeta values. The MZV $\z(n_1,\ldots,n_r)$ is said to have \textit{weight} $n_1+\cdots+n_r$ and \textit{depth} $r$. Considering $\z$ as a function $\qpoly\to\C$ via iterated integrals, as described in \cite{waldschmidt}, we extend these notions to $\qpoly$, where the weight $|w|$ of a word $w\in\{e_0,e_1\}^\times$ is given by its length, and the depth $d(w)$ of $w$ is the number of occurrences of $e_1$.
Thus we define
\begin{equation*}
\begin{split}
\mathcal{W}_n\qpoly&:=\langle w\mid |w|\leq n\rangle_\Q\\
\mathcal{D}_n\qpoly&:=\langle w\mid d(w)\leq n \rangle_\Q
\end{split}
\end{equation*}

These induce increasing filtrations on the space of motivic MZVs and hence decreasing filtrations on $\gmot$, given via its embedding into $\qpoly$ (here viewed as its own graded dual) by
\begin{equation*}
\begin{split}
\mathcal{W}^n\qpoly&:=\langle w\mid |w|\geq n\rangle_\Q\\
\mathcal{D}^n\qpoly&:=\langle w\mid d(w)\geq n \rangle_\Q
\end{split}
\end{equation*}
It is known that weight is a grading for $\gmot$, but depth is not. Euler's well known result
$$\z(1,2)=\z(3)$$
provides an easy example of this, as do any stuffle relations
$$\z(k)\z(l)=\z(k,l)+\z(l,k)+\z(k+l).$$

As such, we can consider relations among depth-graded MZVs to simplify the problem. In \cite{depthgraded}, Brown studies the associated bigraded Lie algebra $\dg:=\gr_\mathcal{D}\gmot$ and its embedding into the linearised double shuffle algebra $\ls$. However, while he finds that  the depth 1 part of $\sigma_{2k+1}$ define canonical representatives of $\sigma_{2k+1}$ in $\dg\to\qpoly$, there also exist relations in $\dg$, and hence `exceptional' generators are needed. These relations are shown to have a somewhat mysterious connection to modular forms by Pollack \cite{pollack}, and this has been further explored by Baumard and Schneps \cite{schneppollack}. 
\begin{example}
In $\dg$, we have the additional relation
$$\{\sigma_3,\sigma_9\}-3\{\sigma_5,\sigma_7\}=0,$$
corresponding to the relation
$$28\z(3,9)+150\z(5,7)+168\z(7,5)=\frac{5197}{691}\z(12).$$
\end{example}
However, finding such relations can be difficult. as such relations cannot be derived using only the double shuffle relations in low depth, and suggests that `depth-graded' multiple zeta values may not be the most natural choice of object to study.

In this paper, we propose an alternative graded analogue of MZVs, based on the work of Charlton \cite{charthesis}.  Unlike the depth graded MZVs, we do not encounter additional relations, and in fact obtain something isomorphic to the vector space of motivic MZVs. Furthermore, not only is this new filtration compatible with the motivic structure, it is equal to the coradical filtration associated to the motivic coaction and is therefore naturally occuring within the motivic formalism, in addition to having a simple combinatorial description. We go on to show that Charlton's cyclic insertion conjecture holds in this setting and additionally show a new family of `block shuffle' relations hold. In addition, we provide a complete description of the image of the generators $\{\sigma_{2k+1}\}$ in the associated graded Lie algebra.

\section{Motivic background}
We now briefly summarise the essential aspects of the motivic formalism. For greater detail, we refer the reader to \cite{motperiods}, \cite{brownmixedtate}, or \cite{delmixedtate}. We first recall the definition of iterated integrals on $\proj$ \cite{chen}.

\begin{defn}
For any sequence $(a_0;a_1,\ldots,a_n;a_{n+1})\in\{0,1\}^{n+2}$, with $a_1=1, a_n=0$, define 
$$\ic(a_0;a_1,\ldots,a_n;a_{n+1}):=\int_{a_0\leq t_1\leq\cdots\leq t_n\leq a_{n+1}}\prod_{k=1}^n \frac{dt_k}{t_k-a_k}$$
\end{defn}

\begin{prop}[Chen \cite{chen}]
There exists a unique extension of $\ic(a_0;a_1,\ldots,a_n;a_{n+1})$ to all sequences such that
\begin{itemize}
\item $\ic(a_0;a_1)=1$
\item $\ic(a_0;0;a_1)=\ic(a_0;1;a_1)=0$
\item Products satisfy \begin{align*}&\ic(a_0;a_{1},\ldots,a_{m};a_{m+n+1})\ic(a_0;a_{m+1},\ldots,a_{m+n};a_N)\\
={}&\sum_{\sigma\in\Sh_{m+n,n}}\ic(a_0;a_{\sigma(1)},\ldots,a_{\sigma(m+n)};a_{m+n+1})\end{align*}
where $$\Sh_{m+n,n}:=\{\sigma\in S_{m+n}\mid \sigma(1)<\sigma(2)<\cdots<\sigma(m),\ \sigma(m+1)<\cdots<\sigma(m+n)\}$$
and $S_N$ denotes the symmetric group on $\{1,\ldots,N\}$.
\end{itemize}
\end{prop}

Every multiple zeta value can be viewed as an iterated integral:
$$\z(k_1,\ldots,k_r)\mapsto (-1)^r \ic(0;1,\{0\}^{k_1-1},\ldots,1,\{0\}^{k_r-1};1)$$
where $\{0\}^k$ is the sequence given by $0$ repeated $k$ times.

\begin{example}\label{zeta2int}
We claim $\z(2)=-\ic(0;1,0;1)$.
\begin{align*}
\ic(0;1,0;1)=&{} \int_{0\leq t_1\leq t_2\leq 1}\frac{dt_1}{t_1-1}\frac{dt_2}{t_2}\\
=&{} -\int_{0\leq t_1\leq t_2\leq 1}dt_1\sum_{n\geq 0}t_1^n \frac{dt_2}{t_2} = -\int_{0\leq t_2\leq 1}dt_2\sum_{n\geq 0} \frac{t_2^n}{n+1}\\
={}& -\sum_{n\geq 0}\frac{1}{(n+1)^2}=\z(2)
\end{align*}
\end{example}

As such, we can view $\ic$ as a linear function on $\qpoly$ by defining, for $w=e_{i_1}\cdots e_{i_n}$,
$$\ic(0;w;1):=\ic(0;i_1,\ldots,i_n;1)$$
and 
$$\ic(w):=\ic(i_1;i_2,\ldots,i_{n-1};i_n).$$
We then lift these to motivic iterated integrals.

Let $\mathcal{M}\mathcal{T}(\mathbb{Z})$ be the category of mixed Tate motives over $\text{Spec}(\mathbb{Z})$. This is the smallest Tannakian category generated by the Tate motive $\Q(-1)=\text{H}^1(\mathbb{G}_m)$, and is equipped with two fibre functors $\omega_B,\omega_{dR}$ to $\text{\textbf{Vec}}_\Q$.

As a Tannakian category, it is equivalent to the category of representations of an affine group scheme, called the motivic Galois group $\Gmot$, the de Rham realisation of which is given by group scheme of tensor automorphisms of $\omega_{dR}$. It is a deep result \cite{delmixedtate} that $\Gmot^{dR}$  decomposes as
$$\Gmot^{dR} = \Umot^{dR}\rtimes\mathbb{G}_m,$$
for a pro-unipotent group scheme $\Umot^{dR}$, whose graded Lie algebra is the motivic Lie algebra:
$$\gmot\cong\Lie[\sigma_3,\sigma_5,\ldots].$$
As mentioned earlier, $\gmot\subset \qpoly$, (in fact $\gmot\subset\Lie[e_0,e_1]$). The Lie algebra structure is not the usual Lie bracket, but is instead given by the Ihara bracket
$$\{\cdot,\cdot\}:\gmot\wedge\gmot\to\gmot$$
defined by the antisymmetrisation of the linearised Ihara action:
\begin{equation*}
\sigma\circ e_0^ne_1u := e_0^n\sigma e_1u - e_0^ne_1 \sigma^* u + e_0^n e_1(\sigma\circ u)
\end{equation*}
where $(a_1\cdots a_n)^*:=(-1)^na_n\cdots a_1$, and $\sigma\circ 1 = 1\circ\sigma=\sigma$.

Define $\mathcal{A}:=\mathcal{O}(\Umot^{dR})$ to be the graded ring of affine functions of the pro-unipotent part. From the work of Goncharov \cite{goncharov} and Brown \cite{brownmixedtate}, this can be lifted to a ring $\mathcal{H}$ generated by symbols $\im(a_0;a_1,\ldots,a_n;a_{n+1})$ called motivic iterated integrals  such that the following hold.

\begin{itemize}
\item There exists an algebra map $\text{per}:\mathcal{H}\to\mathcal{C}$ such that
$$\text{per}:\im(a_0;a_1,\ldots,a_n;a_{n+1})\mapsto\ic(a_0;a_1,\ldots,a_n;a_{n+1}).$$
\item The quotient by the ideal $I_2$ generated by $\im(0;1,0;1)$, $\mathcal{H}/I_2=\mathcal{A}$.
\item $\mathcal{H}$ is a comodule over $\mathcal{A}$, with coaction
$$\Delta:\mathcal{H}\to\mathcal{A}\otimes\mathcal{H}.$$
\item $\mathcal{H}$ is equal to quotient of the free algebra on $\im(a_0;\ldots;a_{n+1})$ by the largest subideal of $\ker(\text{per})$ that is stable under the coaction.
\end{itemize}

Explicitly, the coaction is given by the formula
\begin{equation*}
\begin{split}
\Delta\im(a_0;a_1,\ldots,a_n;a_{n+1}):&= \\
\sum_{0=i_0<i_1<\cdots<i_k<i_{k+1}=n+1}\prod_{p=0}^k \ia(a_{i_p};a_{i_p+1},\ldots,a_{i_{p+1}-1};a_{i_{p+1}})&\otimes \im(a_0;a_{i_1},\ldots,a_{i_k};a_{n+1})
\end{split}
\end{equation*}
where $0\leq k \leq n$, and $\ia(a_0;a_1,\ldots,a_n;a_{n+1})$ is the image of $\im(a_0;a_1,\ldots,a_n;a_{n+1})$ in $\mathcal{A}$.

We define a motivic multiple zeta value by $$\zm(k_1,\ldots,k_r):=(-1)^r\im(0;1,\{0\}^{k_1-1},\ldots,1,\{0\}^{k_r-1};1).$$
We have $\text{per}(\zm(k_1,\ldots,k_r))=\z(k_1,\ldots,k_r)$, and that motivic MZVs satisfy any relations of ``motivic'' or ``geometric'' origin among numerical MZVs. Conjecturally, these all relations among MZVs, and they are encoded by coefficients of elements of $\Gmot$, or equivalently, by the coaction.

\begin{thm}[Brown]
Let $\mathcal{L}:=\mathcal{A}_{>0}/\mathcal{A}_{>0}^2$ and denote by $\pi_{2k+1}$ the projection $\mathcal{A}\to\mathcal{L}_{2k+1}$ onto the weight $2k+1$ part of $\mathcal{L}$ and define $$D_{2k+1}:=(\pi_{2k+1}\otimes \id)\circ\Delta,$$
and let $\mathcal{H}_n$ be the weight $n$ part of $\mathcal{H}$. Then 
$$\mathcal{H}_n\cap\ker\bigoplus_{1<2k+1<n}D_{2k+1} = \Q\zm(n).$$
\end{thm}
\begin{cor}
The Hoffman zeta values
$$\{\zm(k_1,\ldots,k_r)\mid k_i\in\{2,3\}\}$$
forms a basis for $\mathcal{H}$. Thus every MZV can be written as a linear combination of MZVs whose entries are in $\{2,3\}$.
\end{cor}

We will attempt to describe elements of $\gmot$, which is dual to the Lie coalgebra of indecomposables $\mathcal{L}$ via the pairing
$$\langle u, \im(0;v;1)\rangle=\delta_{u,v}\text{ for all }u,v\in\{e_0,e_1\}^\times$$
i.e. the pairing of a word and an motivic iterated integral is $1$ if the underlying word is the same and $0$ otherwise. Via this pairing, elements of $\gmot$ define relations among motivic MZVs, modulo products.

\begin{example}
In weight $5$, $\gmot$ is spanned by
$$\sigma_5 = e_1e_0^4 -\frac{9}{2}e_1e_0^2e_1e_0 +\cdots$$
and
$$\im(0;1,0,0,1,0;1) = -\frac{9}{2}\im(0;1,0,0,0,0;1)-2\im(0;1,0;1)\im(0;1,0,0;1)$$
or equivalently
$$\zm(3,2)=\frac{9}{2}\zm(5)-2\zm(2)\zm(3).$$
\end{example}

\begin{remark}
We take our convention as identifying $e_i\leftrightarrow \frac{dz}{z-i}$. As such, elements of $\gmot$  really describe relations among iterated integrals, rather than multiple zeta values. As such, our results are `depth-signed'.
\end{remark}

This gives a scheme for proving families relations: show they hold for some choice of generators $\{\sigma_{2k+1}\}$, and show that they are preserved by the Ihara bracket.

\section{The block filtration}
In addition to the weight and depth filtrations, we will define a `block filtration' on (motivic) multiple zeta values, arising from the work of Charlton \cite{charthesis}. In his thesis, Charlton defines the block decomposition of a word in two letters $\{x,y\}$ as follows.

Begin by defining a word in $\{x,y\}$ to be \textit{alternating} if it is non-empty and has no subsequences of the form $xx$ or $yy$. There are exactly two alternating words of any given length: one beginning with $x$ and one beginning with $y$. Charlton shows that every non-empty word $w\in\{x,y\}^\times$ can be written uniquely as a minimal concatenation of alternating words. In particular, he defines the block decomposition $w=w_1w_2\cdots w_k$ as the unique factorisation into alternating words such that the last letter of $w_i$ equals the first letter of $w_{i+1}$.

We can use this to define a degree function on words in two letters.

\begin{defn}
Let $w\in\{x,y\}^\times$ be a word of length $n$, given by $w=a_1\cdots a_n$. Define its block degree $\degb(w)$ to be one less than the number of alternating words in its block decomposition. Equivalently, define
$$\degb(w):=\#\{i : 1\leq i < n\text{ such that }a_i=a_{i+1}\}$$
\end{defn}

\begin{remark}
	Note that, unlike depth, the block degree of a word is preserved by the duality anti-homomorphism mapping $e_i\mapsto e_{1-i}$, induced by the automorphism $z\mapsto 1-z$ of $\P^1\setminus\{0,1,\infty\}$.
\end{remark}

We can then define an increasing filtration on $\qpoly$ by
$$\widetilde{\mathcal{B}}_n\qpoly := \langle w\mid \degb(w)\leq n \rangle_\Q$$
which, following the suggestion of Brown in an open letter to Charlton \cite{francisnote}, when restricted to a filtration on $e_1\qpoly e_0$ induces a filtration on motivic multiple zeta values
$$\widetilde{\mathcal{B}}_n\mathcal{H}:=\langle \zm(w)\mid w=e_1ue_0,\ \degb(w)\leq n\rangle_\Q.$$
Brown goes on to show the following.

\begin{prop}[Brown]\label{blockmotivic}
Let $\Gmot^{dR}$ denote the de Rham motivic Galois group of the category $\mathcal{M}\mathcal{T}(\mathbb{Z})$, and let $\Umot^{dR}$ denote its unipotent radical. Then $\widetilde{\mathcal{B}}_n$ is stable under the action of $\Gmot^{dR}$, and $\Umot^{dR}$ acts trivially on $\gr^{\widetilde{\mathcal{B}}} \mathcal{H}$. Equivalently
$$\Delta^r(\widetilde{\mathcal{B}}_n\mathcal{H})\subset\mathcal{O}(\Umot^{dR})\otimes\widetilde{\mathcal{B}}_{n-1}\mathcal{H}$$
where $\Delta^r(x):=\Delta(x)-x\otimes 1 -1\otimes x$ is the reduced coproduct.
\end{prop}

\begin{cor}[Brown]\label{blocklevel}
Brown's block filtration induces the level filtration on the subspace spanned by the Hoffman motivic multiple zeta values $\zm(n_1,\ldots,n_r)$, with $n_i\in\{2,3\}$, where the level is the number of indices equal to $3$.
\end{cor}
\begin{proof}
The word corresponding to $(n_1,\ldots,n_r)$, with $n_i\in\{2,3\}$ of level $m$ has exactly $m$ occurrences of the subsequence $e_0e_0$ and none of $e_1e_1$. Therefore, its block degree is exactly $m$.
\end{proof}

As a corollary to both this and Brown's proof that the Hoffman motivic multiple zeta values form a basis of $\mathcal{H}$, we obtain the following.

\begin{cor}[Brown]
Every element in $\widetilde{\mathcal{B}}_n\mathcal{H}$ of weight $N$ can be written uniquely as a $\Q$-linear combination of motivic Hoffman elements of weight $N$ and level at most $n$. Additionally
$$\sum_{m,\ n\geq 0}\dim\ \gr^{\widetilde{\mathcal{B}}}_m\mathcal{H}_ns^mt^n = \frac{1}{1-t^2-st^3}$$
where $\mathcal{H}_n$ denotes the weight $n$ piece of $\mathcal{H}$.
\end{cor}

However, trying to naively extend this filtration by
$$\widetilde{\mathcal{B}}_n\mathcal{H}=\langle\zm(w)\mid \degb(w)\leq n\rangle_\Q$$
we find that the associated graded $\gr^{\widetilde{\mathcal{B}}}\mathcal{H}$ becomes nearly trivial.
\begin{example}
We have that $\zm(3)=\zm(e_1e_0^{2})\in\widetilde{\mathcal{B}}_{1}$, while $\zm(e_0e_1e_0)\in\widetilde{\mathcal{B}}_0$. As $\zm(3)/\zm(e_0e_1e_0)\in\Q$, this implies that the block-graded version of $\zm(3)$ vanishes.
\end{example}
Comparison of congervent and divergent words leads to similar vanishing of the Hoffman basis. If we instead extend the filtration as follows, we obtain a much more interesting structure.

\begin{defn}\label{trueblockdef}
We define the block filtration of $\qpoly$ by
$$\mathcal{B}_n\qpoly := \langle w\mid \degb(e_0we_1)\leq n\rangle_\Q.$$
This induces the block filtration of motivic multiple zeta values
$$\mathcal{B}_n\mathcal{H} := \langle \zm(w)\mid \degb(0w1)\leq n\rangle_\Q.$$
\end{defn}
 This filtration agrees with our earlier definition if we restrict to $w\in e_1\qpoly e_0$, but the associated graded remains interesting.

\begin{prop}\label{blockfullmotivic}
The block filtration is stable the motivic coaction:
$$\Delta^r\mathcal{B}_n\mathcal{H} \subset \sum_{k=1}^{n-1} \mathcal{B}_{k}\mathcal{A}\otimes \mathcal{B}_{n-k}\mathcal{H}$$
\end{prop}
\begin{proof}
We will in fact show a stronger statement, that $\Delta$ is graded for block degree at the level of words. Let $\mathcal{I}:=\langle \iform(0;w;1)\mid w\in\{0,1\}^\times\rangle_\Q$ be the vector space spanned by formal symbols, with natural projection
\begin{equation*}
\begin{split}
\mathcal{I}&\to\mathcal{H},\\
\iform(0;w;1)&\mapsto \im(0;w;1)
\end{split}
\end{equation*}
and, similarly, a natural projection $\mathcal{I}\to\mathcal{A}$.

Recall that the motivic coaction is essential determined by the infinitesimal coactions, given explicitly by
\begin{equation*}
\begin{split}
D_{2r+1}:\mathcal{H}_N \to \mathcal{L}_{2r+1}&\otimes\mathcal{H}_{N-2r-1}\\
\im(a_0;a_1,\ldots,&a_N;a_{N+1})\mapsto \\
\sum_{p=0}^{N-2r-1}\ia(a_p;a_{p+1},\ldots,a_{p+2r+1};a_{p+2r+2})&\otimes \im(a_0;a_1,\ldots,a_p,a_{p+2r+2},\ldots,a_N;a_{N+1})
\end{split}
\end{equation*}
where $\ia(a_0;a_1,\ldots,a_m;a_m+1)$ is taken to be the projection of $\im(a_0;a_1,\ldots,a_m;a_m+1)$ to $\mathcal{L}$. Note that the coaction and infinitesimal coactions lift to maps $\mathcal{I}\to \mathcal{I}\otimes\mathcal{I}$, by calculating them purely symbolically, and introducing the relations
\begin{align*}
\iform(0;w;0)&=\iform(1;w;1)=0,\\
\iform(1;a_n,\ldots,a_1;0)&=(-1)^n\iform(0;a_1,\ldots,a_n;1).
\end{align*}
 In a slight abuse of notation, we will denote both the motivic coaction and this formal lift by $\Delta$, and the infinitesimal coaction and its formal lift by $D_{2r+1}$

Define $\mathcal{I}_n:= \langle \iform(0;w;1)\mid \degb(0w1)=n\rangle_\Q$. It is sufficient to show that $\Delta\mathcal{I}_n\subset \sum_{i=0}^n\mathcal{I}_i\otimes\mathcal{I}_{n-i}$, as the result follows upon composition with the necessary projections. In fact, it suffices to show that 
$$D_{2r+1}\mathcal{I}_n\subset \sum_{i=0}^n\mathcal{I}_i\otimes\mathcal{I}_{n-i}.$$
Now, consider $\iform(0;w;1)$, $w$ a word in $\{0,1\}$ such that $\degb(0w1)= n$. Then we can decompose $0w1=b_1b_2\cdots b_{n+1}$ into alternating blocks, and consider the action of $D_{2n+1}$ on $\iform(b_1\cdots b_{n+1})$. All terms in $D_{2n+1}\iform(b_1\cdots b_{n+1})$ will be of the form
$$\iform(x;b_i''b_{i+1}\cdots b_{i+j}';y)\otimes \iform(b_1\cdots b_{i-1}b_i'xyb_{i+j}''b_{i+j+1}\cdots b_{n+1})$$
for some $1\leq i \leq {n+1}$, where $b_i = b_i'xb_i''$, $b_{i+j}=b_{i+j}'yb_{i+j}''$. For the left hand term to be non-zero, we must have $x\neq y$, and so we see
\begin{equation*}
\begin{split}
\degb(xb_i''b_{i+1}\cdots b_{i+j}'y) &= j,\\
\degb(b_1\cdots b_i'xyb_{i+j}''\cdots b_{n+1}) &= n - j,
\end{split}
\end{equation*}
by counting the blocks. Thus, we get that the total block degree of any term in the coproduct is $n$, and the result follows.
\end{proof}

\begin{cor}\label{coradical}
The block filtration on $\qpoly$ induces the coradical filtration on $\mathcal{H}$.
\end{cor}

\begin{cor}\label{iharagraded}
The (linearised) Ihara action $\circ:\Lie[e_0,e_1]_{\geq 2}\otimes \qpoly \to \qpoly$ is graded for block degree, where $\Lie[e_0,e_1]_{\geq 2}$ is the space of Lie polynomials of degree at least 2.
\end{cor}
\begin{proof}
The Ihara action is dual to the motivic coaction. As this proof shows the coaction to be, at the level of words, graded for block degree, the claim follows immediately. One can also show this directly via the recursive formula \cite{anatomy} for the linearised Ihara action.
\end{proof}

We also recall a short observation due to Charlton \cite{charthesis}.

\begin{lem}\label{blockparity}
Let $w=w_1\cdots w_n$ be a word in $\{0,1\}^\times$ of length $n$, with $\degb(w)=b$. Then $\im(w)=0$ if $b\equiv w+1\mod2$.
\end{lem}
%

\begin{remark}
This provides a natural analogue of the depth parity theorem \cite{depthgraded}.
\begin{prop}
Suppose $\sigma\in\mathfrak{ls}$ is of weight $N$ and depth $d$. Then, if $N$ and $d$ are of opposite parity, $\sigma=0$. That is, there are no non-trivial solutions to the linearised double shuffle equations with weight and depth of opposite parity.
\end{prop}
With Proposition \ref{canonicalgen}, we obtain a similar corollary to the final conclusion of the following corollary.
\begin{cor}
For a solution to the double shuffle equations mod products $\phi\in\dmr$, of weight  $N$, the depth $d+1\not\equiv N\mod2$ components are uniquely determined by the lower depths. In particular, $\sigma_{2n+1}$ is uniquely determined in depths 1 and 2.
\end{cor}
Specifically, $\sigma_{2n+1}$ is uniquely determined in block degree 1 and 2.
\end{remark}

\section{Block-graded multiple zeta values and an encoding of relations}
As the block filtration is motivic and invariant under the duality arising from the symmetry $z \mapsto 1-z$ of $\P^1\setminus\{0,1,\infty\}$, we can consider the associated graded algebra $\gr^\mathcal{B}\mathcal{A}:=\bigoplus_{n=0}^\infty \mathcal{B}_n\mathcal{A}/\mathcal{B}_{n-1}\mathcal{A}$. We follow the example of Brown's depth-graded multiple zeta values \cite{depthgraded}.

\begin{defn}
Define $\mathcal{B}^n\qpoly:=\langle w\mid \degb(e_0we_1)\geq n\rangle_\Q$ and define $$\gr_\mathcal{B}\gmot := \bigoplus_{n=0}^\infty \mathcal{B}^n\gmot/\mathcal{B}^{n+1}\gmot$$ where we identify $\mathcal{B}^n\gmot/\mathcal{B}^{n+1}\gmot$ with its image in $\mathcal{B}^n\Q\langle e_0,e_1\rangle/\mathcal{B}^{n+1}\Q\langle e_0,e_1\rangle$, equipped with the block-graded Ihara bracket.
\end{defn}

\begin{defn}
If  $\degb(e_0we_1)=n$, define $\ib(0;w;1)$ to be the image of $\ia(0;w;1)$ in $\mathcal{B}_n\mathcal{A}/\mathcal{B}_{n-1}\mathcal{A}$. Similarly, define $\ibl(0;w;1)$ to be the image of $\il(0;w;1)$ in $\mathcal{B}_n\mathcal{L}/\mathcal{B}_{n-1}\mathcal{L}$. Define $\zb$ and $\zbl$ similarly.
\end{defn}

\begin{defn}
Fix an embedding of $\{\sigma_3,\sigma_5,\ldots\}\hookrightarrow \qpoly$. We define the block-graded generators $\{p_{2k+1}\}_{k\geq 1}$ to be the image of the generators $\{\sigma_{2k+1}\}_{k\geq 1}$ of $\gmot$ in $\mathcal{B}^1\qpoly/\mathcal{B}^2\qpoly$. We define the bigraded Lie algebra $\bg$ to be the Lie algebra generated by $p_{2k+1}$ and the Ihara bracket.
\end{defn}

One of the challenges in studying $\gmot$ is that we have an ambiguity in our representation of the generators: $\sigma_{2k+1}$ is unique only up to addition of another element of weight $2k+1$. Its depth one part is canonical, so Brown's depth-graded Lie algebra avoids this issue. We find similar success here.

\begin{prop}\label{canonicalgen}
The generators $p_{2k+1}$ of $\bg$ are canonical, i.e. independent of our choice of embedding of generators $\{\sigma_{2k+1}\}\hookrightarrow \qpoly$.
\end{prop}
\begin{proof}
Let $\sigma_{2k+1}$, $\sigma_{2k+1}^\prime\in \qpoly$ be two choices of generator for $\gmot$ in weight $2k+1$. We must have
$$\sigma_{2k+1}-\sigma_{2k+1}^\prime\in \{\gmot,\gmot\}.$$
Proposition \ref{iharagraded} tells us that the Ihara action is compatible with the block filtration, and so
$$ \{\gmot,\gmot\}\subset \mathcal{B}^2\gmot$$
and therefore
$$p_{2k+1}-p_{2k+1}^\prime= 0.$$
\end{proof}

Note that we can still define a concept of depth on $\bg$ as before. We define the depth of a word $w$ to be $d(w)$, and induce a decreasing filtration on $\bg$ via its embedding $\bg\hookrightarrow\qpoly$. It is interesting here that depth grading gives canonical generators in depth 1, while block grading gives $p_{2k+1}$ consisting only of terms of depth $k$ or $k+1$.

\begin{lem}\label{generatordepth}
The block-graded generators $p_{2k+1}$ contains only depth $k$ and $k+1$ terms.
\end{lem}
\begin{proof}
Suppose $w$ is a word of block degree 1 and weight $2k+1$. Then $e_0we_1$ has two blocks and hence contains exactly one of $e_0^2$ or $e_1^2$. In the first case, the number of $e_1$ must be exactly half of $2k+1 - 1$, i.e. $k$. In the second case, the number of $e_0$ must similarly be $k$ and hence the number of $e_1$ is $k+1$.
\end{proof}

\begin{thm}\label{gradedisbg}
The Lie algebra $\bg$ is freely generated by $\{p_{2k+1}\}_{k\geq 1}$.
\end{thm}
\begin{proof}
We have a bijection between the generators of $\gmot$ and of $\bg$, and Corollary \ref{iharagraded} tells us that the Ihara action is graded for block degree. Thus, we can write an element $$\{p_{2k_1+1},\{\ldots,\{p_{2k_{b-1}+1},p_{2k_b+1}\}\ldots\}\}$$ as the image of $$\{\sigma_{2k_1+1},\{\ldots,\{\sigma_{2k_{b-1}+1},\sigma_{2k_b+1}\}\ldots\}\}$$ in $\mathcal{B}^b\gmot/\mathcal{B}^{b+1}\gmot$. Hence, we have a relation in $\bg$ if and only if the corresponding sum of terms is 0 in $\gr_{\mathcal{B}}\gmot$. Indeed, we have an injective Lie algebra homomorphism $\bg\hookrightarrow\gr_\mathcal{B}\gmot$ induced by the bijection $\{\sigma_{2k+1}\}_{k\geq 1}\leftrightarrow\{p_{2k+1}\}_{k\geq 1}$. 
Now, as $\gr_{\mathcal{B}}\gmot$ is dual to $\gr^{\mathcal{B}}\mathcal{L}$, the existence of relations in $\bg$ implies the existence of additional relations in $\gr^{\mathcal{B}}\mathcal{L}$. To be precise, we must have that $$\dim\ \gr^\mathcal{B}_n\mathcal{L}_N<\dim\ \langle \il(w) \mid \degb(w)=n,\ |w|=N\rangle_\Q.$$
Then, by the proof of Theorem 7.4 in \cite{brownmixedtate}, we know that the right hand side has is spanned by $\{\za(k_1,\ldots,k_r)\}$, where $k_i\in\{3,2\}$,  and $k_i=3$ exactly $n$ times and $k_1+\cdots+k_r=N$. In particular, it has a basis given by $\za(k_1,\ldots,k_r)$ such that $(k_1,\ldots,k_r)$ is a Lyndon word with respect to the order $3<2$. This basis, called the Hoffman-Lyndon basis, forms a spanning set for $\gr^\mathcal{B}_n\mathcal{L}_N$. Thus,
$$\dim\ \gr^\mathcal{B}_n\mathcal{L}_N<\dim\ \langle \ia(w) \mid \degb(w)=n,\ |w|=N\rangle_\Q$$
which implies that there is a sum of Hoffman-Lyndon elements of weight $N$ with $n$ threes that can be written as a sum of Hoffman-Lyndon elements of weight $N$ with fewer threes. However, the Hoffman-Lyndon elements of weight $N$ form a basis of $\mathcal{L}_N$, and, so, no such relation can exist. Thus, we must have that $\mathcal{L}\equiv\gr^\mathcal{B}\mathcal{L}$ as they have equal dimensions, and hence, $\gr_\mathcal{B}\gmot\equiv \gmot$. This implies $\gr_\mathcal{B}\gmot$ and $\bg$ are both freely generated and isomorphic.
\end{proof}

\begin{remark}
While both Brown's $\mathfrak{dg}$ and our $\bg$ have canonical generators, Theorem \ref{gradedisbg} tells us that $\bg$ is free, while there exist relations in $\dg$, and hence `exceptional' generators are needed, first appearing in depth four. These relations are shown to have a somewhat mysterious connection to modular forms by Pollack \cite{pollack}, and this has been further explored by Baumard and Schneps \cite{schneppollack}. However, it is a computationally challenging task, and suggests that `depth-graded' multiple zeta values may not be the most natural choice of object to study.
\end{remark}

\section{Polynomial representations}
We now reframe this Lie algebra in terms of commutative polynomials, similarly to Brown \cite{anatomy}\cite{zeta3} and \'Ecalle \cite{ecalleassoc}, as follows.

Recall that Charlton shows that every word $w\in\{ e_0,e_1\}^\times$ can be written uniquely as a sequence of alternating blocks \cite{charthesis}. In doing so, he establishes a bijection
\begin{equation*}
\begin{split}
\text{bl}:\{e_0,e_1\}^\times\setminus\{\emptyset\} &\to \bigcup_{n=1}^\infty \{0,1\}\times \mathbb{N}^n\\
w&\mapsto (\epsilon;l_1,l_2,\ldots,l_n)
\end{split}
\end{equation*}
where $\epsilon$ defines the first letter of $w$, and $l_1,\ldots,l_n$ describe the length of the alternating blocks.

\begin{example}
These words correspond to the following sequences.
\begin{equation*}
\begin{split}
e_0e_1e_0e_0e_1e_0e_1e_1&\mapsto (0;3,4,1),\\
e_1e_1e_0e_1e_0e_1e_1e_0e_0&\mapsto (1;1,5,2,1)
\end{split}
\end{equation*}
\end{example}

We can use this to define an injection of  vector spaces by
\begin{equation}\label{isomorph}
\begin{split}
\pi_\text{bl}:\qpoly\setminus\{\Q\cdot 1\}&\to \bigoplus_{n=1}^\infty x_1\Q[x_1,\ldots,x_n]x_n\\
w&\mapsto x_1^{l_1}\ldots x_n^{l_n}
\end{split}
\end{equation}
where $\text{bl}(e_0we_1)=(0;l_1,\ldots,l_n)$.

\begin{remark} 
In a slight abuse of notation, it will be convenient for us to identify words $w$ with the sequence $(l_1,\ldots,l_n)$ given by $\text{bl}(e_0we_1)$. In particular, for any $\bullet\in\{\mathfrak{m},\mathfrak{a},\mathfrak{l},\mathfrak{b},\mathfrak{bl}\}$, we define
$$\text{I}^\bullet(l_1,\ldots,l_n):=\text{I}^\bullet(0;w;1)$$
where $w$ is the unique word such that $\text{bl}(e_0 we_1)=(0;l_1,\ldots,l_n)$. If no such word exists, we define $\text{I}^\bullet(l_1,\ldots,l_n):=0$.
\end{remark}

In this formulation, a word of block degree $n$ and weight $N\geq 1$ is  represented by a polynomial in $n+1$ variables of degree $N+2$. From this point on, we shall freely identify elements of $\bg$ with their images under this injection.

\begin{prop}\label{generatordef}
The projections of the depth-signed $\sigma_{2k+1}\in\gmot$ onto their block degree one part are given by
$$p_{2k+1}(x_1,x_2) = q_{2k+1}(x_1,x_2) - q_{2k+1}(x_2,x_1)$$
where
\begin{equation*}
q_{2k+1}(x_1,x_2)=\sum_{i=1}^{k} \left[\binom{2k}{2i} - \left(1-\frac{1}{2^{2k}}\right)\binom{2k}{2k+1-2i}\right]x_1^{2i+1}x_2^{2k+2-2i} -  x_1x_2^{2k+2}
\end{equation*}
and $\sigma_{2k+1}$ have been normalised to correspond to $\frac{(-1)^k}{2}\zeta(2k+1)$.
\end{prop}

\begin{proof}
We will compute the block degree 1 part of $\sigma_{2k+1}$ consisting of terms containing an $e_0^2$. This will give $q_{2k+1}$. That $p_{2k+1}(x_1,x_2)=q_{2k+1}(x_1,x_2)-q_{2k+1}(x_2,x_1)$ follows from duality. In terms of $e_0,\ e_1$, we have
$$q_{2k+1}= \sum_{i=0}^k c_i(e_1e_0)^ie_0(e_1e_0)^{k-i},$$
where $\zm(\{2\}^{i-1},3,\{2\}^{k-i}) = \alpha c_i\zm(2k+1) \text{ (mod }\zm(2)\text{)}$, for $i>0$ and some $\alpha\in\Q$, and $c_0$ is obtained via shuffle regularisation \cite{brownmixedtate}.

Shuffle regularisation of $e_0e_1\cdots e_0$ tells us that
$$c_0 + 2\sum_{i=1}^k c_i = 0.$$
Next, from the work of Zagier \cite{zaghoff},
$$\zeta(\{2\}^{a},3,\{2\}^b) = 2\sum_{r=1}^{a+b+1}(-1)^r\left[\binom{2r}{2a+2}-(1-\frac{1}{2^{2r}})\binom{2r}{2b+1}\right]\zeta(\{2\}^{a+b-r+1})\zeta(2r+1),$$
where $\{c\}^a$ denotes the string $c,c,\ldots,c$ with $c$ repeated $a$ times.
Brown shows in \cite{brownmixedtate} Theorem 4.3 that this lifts to an identity among motivic multiple zeta values. Considered modulo $\zm(2)$, we find
$$\zm(\{2\}^{i-1},3,\{2\}^{k-i})  = 2(-1)^{k}\left[\binom{2k}{2i} - \left(1-\frac{1}{2^{2k}}\right)\binom{2k}{2k+1-2i}\right]\zm(2k+1),$$
and thus, we can take $c_i=\left[\binom{2k}{2i} - (1-\frac{1}{2^{2k}})\binom{2k}{2k+1-2i}\right]$ for $i>0$. The result then follows.
\end{proof}

Computing these sums explicitly, we obtain the following theorem.

\begin{thm}\label{summedgen}
The polynomial $p_{2k+1}(x_1,x_2)$ is given explicitly by the formula
$$p_{2k+1}(x_1,x_2)= x_1x_2(x_1-x_2)\left(\frac{(1-2^{2k+1})(x_1+x_2)^{2k}-(x_1-x_2)^{2k}}{2^{2k}}\right).$$
\end{thm}

With this in mind, we can provide a characterisation of these generators in terms of polynomial equations.

\begin{cor}\label{newgendef}
The polynomial $p_{2k+1}(x_1,x_2)$ is, up to rescaling, the unique homogeneous polynomial $p(x_1,x_2)$ of degree $2k+3$ such that
$$p(x_1,0)=p(0,x_2)=p(x_1,x_2)+p(x_2,x_1)=0,$$
and, defining $r(x_1,x_2):=\frac{p(x_1,x_2)}{x_1x_2(x_1-x_2)}$, satisfying
$$r(0,x)=2r(x,-x),$$
and
$$\left(\partialx{1}\right)^2 r(x_1,x_2)=\left(\partialx{2}\right)^2r(x_1,x_2).$$
\end{cor}

\begin{proof}
From the conditions $p(x_1,0)=p(0,x_2)=p(x_1,x_2)+p(x_2,x_1)=0$, we can write $p(x_1,x_2)=x_1x_2(x_1-x_2)r(x_1,x_2)$. Letting $u=x_1+x_2$, and $v=x_1-x_2$, we can rewrite 
$$\left(\partialx{1}\right)^2 r(x_1,x_2)=\left(\partialx{2}\right)^2r(x_1,x_2)\Leftrightarrow \frac{\partial^2 r}{\partial u\partial v}(u,v)=0,$$
which has polynomial solution, homogeneous of degree $(2k+3)-3=2k$
$$r(u,v)=\alpha u^{2k} + \beta v^{2k}$$
which is to say
$$r(x_1,x_2)=\alpha(x_1+x_2)^{2k}+\beta(x_1-x_2)^{2k}.$$
Finally, the condition
$$r(0,x)=2r(x,-x)$$
gives
$$(\alpha + \beta)x^{2k} = 2^{2k+1}\beta x^{2k},$$
and hence
$$\alpha = -(1-2^{2k+1})\beta,$$
giving the desired result.
\end{proof}

We can provide an exact polynomial formula for the Ihara action. Recall that we have chosen $\gmot$ to differ from Brown's by sending $e_1\mapsto -e_1$, and so this is only accurate for `depth-signed' elements. We delay the proof of this until later.

\begin{thm}\label{Iharanice}
For (depth-signed) elements of the motivic Lie algebra, the Ihara action is given at the level of block-polynomials by
\begin{equation}\label{prettyIhara}
\begin{split}
(f\circ g) (x_1,\ldots,x_{m+n-1}) ={}& (-1)^{(m+1)(n+1)}\sum_{i=1}^{n}\frac{f (x_i,x_{i+1},\ldots,x_{i+m-1})}{x_i-x_{i+m-1}}\\
&\times \left( \frac{1}{x_i}g (x_1,\ldots,x_{i-1},x_i,x_{i+m},\ldots,x_{m+n-1})\right.\\
&\quad -\left.\frac{1}{x_{i+m-1}}g (x_1,\ldots,x_{i-1},x_{i+m-1},\ldots,x_{m+n-1})\right).
\end{split}
\end{equation}

\end{thm}
\section{Relations arising in the polynomial representation}
We find several relations arising naturally in the polynomial representation which are preserved by the Ihara action, and dual to relations in $\gr^\mathcal{B}\mathcal{L}$. We prove these by inducting on block degree. To illustrate this method, we will reprove the duality identity.

\begin{prop}\label{duality}
For all $f(x_1,\ldots,x_n)\in\bg$,
$$f(x_1,\ldots,x_n)=(-1)^{n+1}f(x_n,\ldots,x_1).$$ 
\end{prop}
\begin{proof}
It suffices to show that this holds for $p_{2k+1}$, and that, if this holds for $f,g\in \bg$, then it holds for $f\circ g$. The former holds by definition of $p_{2k+1}$. To see the latter, note that
\begingroup
\allowdisplaybreaks
\begin{align*}
(f\circ g) &(x_{m+n-1},\ldots,x_{1})\\
 ={} &(-1)^{(m+1)(n+1)}\sum_{i=1}^{n}\frac{f (x_{m+n-i},x_{m+n-i-1},\ldots,x_{n+1-i})}{x_{m+n-i}^2-x_{n+1-i}^2}\\
&\times\left(\left(1 + \frac{x_{n+1-i}}{x_{m+n-i}}\right)g (x_{m+n-1},\ldots,x_{m+n-i+1},x_{m+n-i},x_{n-i},\ldots,x_{1})\right.\\
&\quad -\left. \left(1+\frac{x_{m+n-i}}{x_{n+1-i}}\right)g (x_{m+n-1},\ldots,x_{m+n-i+1},x_{n+1-i},\ldots,x_{1})\right)\\
= {}&(-1)^{(m+1)(n+1)}\sum_{i=1}^{n}(-1)^m\frac{f (x_{n+1-i},x_{n+2-1},\ldots,x_{m+n-i})}{x_{n+1-i}^2-x_{m+n-i}^2}\\
&\times\left((-1)^{n+1}\left(1 + \frac{x_{n+1-i}}{x_{m+n-i}}\right)g (x_{1},\ldots,x_{n-i},x_{m+n-i},x_{m+n-i},\ldots,x_{m+n-1})\right.\\
&\quad -\left. (-1)^{n+1}\left(1+\frac{x_{m+n-i}}{x_{n+1-i}}\right)g (x_{1},\ldots,x_{n+1-i},x_{m+n-i+1},\ldots,x_{m+n-1})\right)\\
={} &(-1)^{m+n} (-1)^{(m+1)(n+1)}\sum_{i=1}^{n}\frac{f (x_{n+1-i},x_{n+2-1},\ldots,x_{m+n-i})}{x_{n+1-i}^2-x_{m+n-i}^2}\\
&\times\left( \left(1+\frac{x_{m+n-i}}{x_{n+1-i}}\right)g (x_{1},\ldots,x_{n+1-i},x_{m+n-i+1},\ldots,x_{m+n-1})\right.\\
&\quad -\left.\left(1 + \frac{x_{n+1-i}}{x_{m+n-i}}\right)g (x_{1},\ldots,x_{n-i},x_{m+n-i},x_{m+n-i},\ldots,x_{m+n-1})\right)\\
={} &(-1)^{m+n}(f\circ g)(x_1,\ldots,x_{m+n-1}),
\end{align*}
\endgroup
and hence, the duality relation is preserved by the Ihara bracket.
\end{proof}

We can similarly prove Charlton's cyclic insertion conjecture, up to terms of lower block degree. We will instead show that a more general relation holds, of which cyclic insertion is a corollary. These are the `block shuffle' relations.

\begin{defn}
For any $1\leq r \leq n$, define the shuffle set

$$\Sh_{n,r}=\{\sigma\in S_n \mid \sigma^{-1}(1)<\cdots<\sigma^{-1}(r);\sigma^{-1}(r+1)<\cdots<\sigma^{-1}(n)\}.$$

Then, for any $f\in\Q[x_1,\ldots,x_n]$, define
$$f(x_1\cdots x_r\sh x_{r+1}\cdots x_n):=\sum_{\sigma\in \Sh_{n,r}}f(x_{\sigma(1)},x_{\sigma(2)},\ldots,x_{\sigma(n)}).$$
\end{defn}

\begin{thm}\label{blockshuffle}
For any $f(x_1,\ldots,x_n)\in\bg$, and any $1\leq r <n$, we have 
$$f(x_1x_2\cdots x_r\sh x_{r+1}\cdots x_n) = 0.$$
\end{thm}
\begin{proof}
For $p_{2k+1}$, this is equivalent to $p(x_1,x_2)+p(x_2,x_1)=0$, given by Proposition \ref{duality}. Then, as the Ihara action is associative, it in fact suffices to show that
$$(f\circ g)(x_1\ldots x_r\sh x_{r+1}\ldots x_{n+1}) = 0$$
for all $f=p_{2k+1}(x_1,x_2)$, $g(x_1,\ldots,x_n)\in\bg$.

 We write $(f\circ g)(x_1\ldots x_r\sh x_{r+1}\ldots x_{n+1})$ as
\begin{equation*}
\begin{split}
\sum_{\sigma\in \Sh_{n+1,r}}\sum_{i=1}^n\frac{f(x_{\sigma(i)},x_{\sigma(i+1)})}{x_{\sigma(i)}-x_{\sigma(i+1)}}&\times\\
\left(\frac{g(x_{\sigma(1)},\ldots,x_{\sigma(i)},x_{\sigma(i+2)},\ldots,x_{\sigma(n+1)})}{x_{\sigma(i)}}\right. &-\left.\frac{g(x_{\sigma(1)},\ldots,x_{\sigma(i-1)},x_{\sigma(i+1)},\ldots,x_{\sigma(n+1)})}{x_{\sigma(i+1)}}\right).
\end{split}
\end{equation*}

This sum splits as follows
\begin{equation*}
\begin{split}
\sum_{\sigma\in \Sh_{n+1,r}}\sum_{i=1}^{r-1}\frac{f(x_{\sigma(i)},x_{\sigma(i+1)})}{x_{\sigma(i)}-x_{\sigma(i+1)}}&\times\\
\left(\frac{g(x_{\sigma(1)},\ldots,x_{\sigma(i)},x_{\sigma(i+2)},\ldots,x_{\sigma(n+1)})}{x_{\sigma(i)}}\right. &-\left.\frac{g(x_{\sigma(1)},\ldots,x_{\sigma(i-1)},x_{\sigma(i+1)},\ldots,x_{\sigma(n+1)})}{x_{\sigma(i+1)}}\right)\\  
+\sum_{\sigma\in \Sh_{n+1,r}}\sum_{i=r+1}^n\frac{f(x_{\sigma(i)},x_{\sigma(i+1)})}{x_{\sigma(i)}-x_{\sigma(i+1)}}&\times\\
\left(\frac{g(x_{\sigma(1)},\ldots,x_{\sigma(i)},x_{\sigma(i+2)},\ldots,x_{\sigma(n+1)})}{x_{\sigma(i)}}\right. &-\left.\frac{g(x_{\sigma(1)},\ldots,x_{\sigma(i-1)},x_{\sigma(i+1)},\ldots,x_{\sigma(n+1)})}{x_{\sigma(i+1)}}\right)\\
+\sum_{\substack{\sigma\in \Sh_{n+1,r}\\
\text{such that}\\
\{\sigma(r),\sigma(r+1)\}\neq\{r,r+1\}}}\frac{f(x_{\sigma(r)},x_{\sigma(r+1)})}{x_{\sigma(r)}-x_{\sigma(r+1)}}&\times\\
\left(\frac{g(x_{\sigma(1)},\ldots,x_{\sigma(r)},x_{\sigma(r+2)},\ldots,x_{\sigma(n+1)})}{x_{\sigma(r)}}\right. &-\left.\frac{g(x_{\sigma(1)},\ldots,x_{\sigma(r-1)},x_{\sigma(r+1)},\ldots,x_{\sigma(n+1)})}{x_{\sigma(r+1)}}\right).
\end{split}
\end{equation*}

Denote the first sum by $A$, the second by $B$, and the third by $C$. Now, this sum can be written uniquely as
$$\sum_{1\leq k < l \leq n+1}\frac{f(x_k,x_l)}{x_k-x_l}\left(\frac{G_{k,l}}{x_k}-\frac{H_{k,l}}{x_l}\right)$$
where $G_{k,l}$, $H_{k,l}$ are polynomials related by swapping $x_k\leftrightarrow x_l$.
We have 4 cases to consider
\begin{enumerate}
    \item $l\leq r$,
    \item $k\geq r+1$,
    \item $k<r<r+1<l$,
    \item $k=r=l-1$.
\end{enumerate}

In the first case, both $A$ and $B$ only contribute non-zero terms if $l=k+1$, while $C$ only contributes if $l>k+1$. Thus, denoting by $\Phi_k(\sigma,i)$ the condition $\{\sigma(i)=k,\ \sigma(i+1)=k+1\}$, we have
\begin{equation*}
    \begin{split}
        G_{k,k+1}= {} &\sum_{k\leq i<r}\sum_{\substack{ \sigma\in \Sh_{n+1,r}\\
\text{such that } \Phi_k(\sigma,i)}} g(x_{\sigma(1)},\ldots,x_{\sigma(i-1)},x_k,x_{\sigma(i+2)},\ldots,x_{\sigma(n+1)})\\
        &+\sum_{i>r}\sum_{\substack{ \sigma\in \Sh_{n+1,r}\\
\text{such that }\Phi_k(\sigma,i)}} g(x_{\sigma(1)},\ldots,x_{\sigma(i-1)},x_k,x_{\sigma(i+2)},\ldots,x_{\sigma(n+1)})\\
        &+\sum_{\substack{\sigma\in \Sh_{n+1,r}\\
\text{such that } \Phi_k(\sigma,r)}} g(x_{\sigma(1)},\ldots,x_{\sigma(i-1)},x_k,x_{\sigma(i+2)},\ldots,x_{\sigma(n+1)}).\\
    \end{split}
\end{equation*}
Let $P(\sigma,r)$ denote the condition
$$\{\sigma^{-1}(1)<\cdots<\sigma^{-1}(r);\sigma^{-1}(r+1)<\cdots<\sigma^{-1}(n+1)\}.$$
Then, this is a sum over the set of permutations
\begin{equation*}
        \bigcup_{k\leq i <r}\{\sigma\mid\Phi_k(\sigma,i)\text{ and } P(\sigma,r)\} \cup\bigcup_{i>r}\{\sigma\mid\Phi_k(\sigma,i)\text{ and } P(\sigma,r)\}
\end{equation*}
which is clearly in bijection with a set of shuffles of $[n+1]\setminus\{k+1\}$, and so the contribution is $0$ by induction.

Then, if $l>k+1$, we find that the non-zero terms in
$$\frac{f(x_k,x_l)}{x_k-x_l}\left(\frac{G_{k,l}}{x_k}-\frac{H_{k,l}}{x_l}\right)$$ 
arising from permutations with $\sigma(r)=k,\sigma(r+1)=l$ cancel with those arising from $\sigma(r)=l,\sigma(r+1)=k$. Thus, in this case, 
$$\frac{f(x_k,x_l)}{x_k-x_l}\left(\frac{G_{k,l}}{x_k}-\frac{H_{k,l}}{x_l}\right)=0.$$ 

The second case, $k\geq r+1$, is similar. In the third case, every term arising from a permutation with $\sigma(i)=k,\sigma(i+1)=l$ cancels with the term arising from the permutation $\tau_{k,l}\circ\sigma$, where $\tau_{k,l}$ is the transposition $(k\ l)$.

Finally, in the fourth case, our sum splits into a sum over the following sets
$$\bigcup_{i<r}\{\sigma\in \Sh_{n+1,r} \mid \sigma(i)=r,\sigma(i+1)=r+1\},$$
$$\bigcup_{i>r}\{\sigma\in \Sh_{n+1,r} \mid \sigma(i)=r,\sigma(i+1)=r+1\},$$
$$\bigcup_{i<r}\{\sigma\in \Sh_{n+1,r} \mid \sigma(i)=r+1,\sigma(i+1)=r\},$$
$$\bigcup_{i>r}\{\sigma\in \Sh_{n+1,r} \mid \sigma(i)=r+1,\sigma(i+1)=r\}.$$
All of these must be empty due to the order preserving property of shuffle permutations. Thus,
$$(f\circ g)(x_1\cdots x_r\sh x_{r+1}\cdots x_{n+1}) = 0.$$
\end{proof}

\begin{cor}
For any finite sequence of integers $l_1,\ldots,l_n$, and any $1\leq r < n$, we have
$$\sum_{\sigma\in\Sh_{n,r}}\ibl(l_{\sigma(1)},\ldots,l_{\sigma(n)})=0$$
when considered modulo products.
\end{cor}

\begin{proof}
Using Theorem \ref{gradedisbg}, we can consider $\bg$ as the dual Lie algebra to the graded Lie coalgebra of indecomposables $\gr^\mathcal{B}\mathcal{L}$, and hence, relations among the coefficients of elements of $\bg$ induce relations among elements of $\gr^\mathcal{B}\mathcal{L}$. Specifically, we define a $\Q$-linear pairing
$$\langle \ibl(l_1,\ldots,l_n) \mid x_1^{k_1}\ldots x_m^{k_m}\rangle := \delta_{n.m}\delta_{l_1,k_1}\ldots\delta_{l_n,k_n}$$
where $\ibl(l_1,\ldots,l_n)$ is the image of $\ib(l_1,\ldots,l_n)$ in $\gr_\mathcal{B}\mathcal{L}$. We have that $R$ is a relation in $\gr^\mathcal{B}\mathcal{L}$ if and only if $\langle R \mid f\rangle=0$ for all $f\in\bg$. Hence, as $f(x_1x_2\cdots x_r\sh x_{r+1}\cdots x_n) = 0$ for all $f\in\bg$, we must have that
$$\sum_{\sigma\in\Sh_{n,r}}\ibl(l_{\sigma(1)},\ldots,l_{\sigma(n)})=0.$$
\end{proof}

\begin{cor}[Block-graded cyclic insertion]\label{cyclic}
Let $\Cyc_n$ be the cyclic group generated by the permutation $(1\ 2\cdots n)$. Then the cyclic sum vanishes:
$$\sum_{\sigma\in\Cyc_n}\ibl(l_{\sigma(1)},l_{\sigma(2)},\ldots,l_{\sigma(n)}) = 0,$$
\end{cor}
\begin{proof}
It suffices to show that
$$\sum_{\sigma\in\Cyc_n} f(x_{\sigma(1)},x_{\sigma(2)},\ldots,x_{\sigma(n)})=0$$
for all $f\in\bg$.

Suppose $f\in\bg$. Then, Theorem \ref{blockshuffle} implies that the image of $f$ under the following vector space isomorphism
\begin{equation}\label{zalphabet}
\begin{split}
\bigoplus_{n=0}^\infty \Q[x_1,\ldots,x_n] &\xrightarrow{\sim} \Q\langle z_1,z_2,z_3,\ldots\rangle\\
x_1^{i_1} x_2^{i_2}\ldots x_n^{i_n} &\mapsto z_{i_1}z_{i_2}\cdots z_{i_n}
\end{split}
\end{equation}
lies in $\Lie[z_1,z_2,\ldots]$. In particular, the image lies in the span of elements of degree at least 2. Now, we define a linear map $\mathcal{C}:\Q\langle z_1,z_2,\ldots\rangle\to \Q\langle z_1,z_2,\ldots\rangle$ by
$$\mathcal{C}(z_{i_1}z_{i_2}\cdots z_{i_n}) = \sum_{\sigma\in\Cyc_n}z_{i_{\sigma(1)}}z_{i_{\sigma(2)}}\cdots z_{i_{\sigma(n)}}$$
for a word of length $n$. Thus, it suffices to show that $\mathcal{C}(Z)=0$ for all $Z\in\Lie[z_1,z_2,\ldots]$ of degree at least 2.

Note, for any monomials $X,\ Y$ in $\{z_1,z_2,\ldots\}$ of degree $k,\ n-k$ respectively, we have $[X,Y] = XY - \sigma(XY)$, for some $\sigma\in\Cyc_n$ acting by cyclic rotations on words of length $n$. Thus,
$$\mathcal{C}([X,Y])=\mathcal{C}(XY)-\mathcal{C}(\sigma(XY))=\mathcal{C}(XY)-\mathcal{C}(XY)=0$$
and so the image of  any element of degree at least two in $\Lie[z_1,z_2,\ldots]$ is zero, and hence 
$$\sum_{\sigma\in\Cyc_n} f(x_{\sigma(1)},x_{\sigma(2)},\ldots,x_{\sigma(n)})=0$$
\end{proof}

\begin{remark}
As in this proof, it can be useful to consider $\bg$ as a subspace of the Hopf $ \Q\langle z_1,z_2,z_3,\ldots\rangle$, with the standard concatenation product, and a coproduct given by $\Delta z_i =z_i\otimes 1 + 1\otimes z_i$. For example, Theorem \ref{blockshuffle} implies elements of $\bg$ are primitive for this coproduct, we immediately obtain Proposition \ref{duality} as a corollary, by considering the antipode map, i.e. the antihomomorphism $z_i\mapsto -z_i$. This is an idea explored further in the author's thesis \cite{mythesis}.
\end{remark}

\begin{example}
The cyclic insertion conjecture is expected to hold, modulo $\zm(2)$ in the ungraded setting as well. For example, one can check that the following hold exactly, up to powers of $\zm(2)$.
\begin{equation*}
\begin{split}
\im(4,3,2,2)+\im(3,2,2,4)+\im(2,2,4,3)+\im(2,4,3,2) &=0\\
-\zm(2,1,2,1,3)+\zm(3,1,3,2)-\zm(1,3,2,1,2)+\zm(1,2,3,3)&=0
\end{split}
\end{equation*}
The block shuffle relation has additional terms of lower order. For example, the following hold exactly.
\begin{equation*}
\begin{split}
\im(4,3,2,2)+\im(3,4,2,2)+\im(3,2,4,2)+\im(3,2,2,4)&=\im(9,3)+\im(3,8)\\
-\zm(2,1,2,1,3)+\zm(3,2,1,3)+\zm(3,1,2,3)+\zm(3,1,3,2)&=\zm(2,2,2,3)+\zm(3,2,2,2)
\end{split}
\end{equation*}
An explicit general form for the ungraded block shuffle has been conjectured by Hirose and Sato, and is discussed in greater detail in Chapter 4 of the author's thesis \cite{mythesis}.\end{example}

\section{Shuffle Regularisation}

The double shuffle relations among iterated integrals are not, in general, compatible with the block filtration. However, the regularisation relation obtained by shuffling with an element of weight 1, does respect the block filtration.

\begin{thm}\label{weight1shuff}
Let $\pi_1:\qpoly\to\Q e_0\oplus\Q e_1$ denote the projection map onto weight 1 , and let $\Delta:\qpoly\to\qpoly\otimes\qpoly$ be the coproduct defined by $\Delta(e_i)=e_i\otimes 1 + 1\otimes e_i$. The map $\Delta_1:=(\pi_1\otimes\id)\Delta$ is compatible with the block filtration:
$$\Delta_1\mathcal{B}^n\qpoly\subset\mathcal{B}^1\qpoly\otimes\mathcal{B}^{n-1}\qpoly.$$
\end{thm}
\begin{proof}
For $w\in\qpoly$ every term in $\Delta_1(w)$ is of the form $e_i\otimes \bar{w}$ for $i\in\{0,1\}$, where $\bar{w}$ is obtained from $w$ by omitting a letter. The left hand side is of block degree 1. The right hand side is of higher block degree, if the omitted letter was internal to a block, and of block degree 1 lower than $w$, if the omitted letter was at the beginning or end of a block.
\end{proof}

Thus, we can take the associated graded map of $\Delta_1$.
\begin{cor}\label{gradedshuffle}
The graded version of $\Delta_1$ annihilates $\bg$:
$\gr_\mathcal{B}(\Delta_1)(\bg)=0.$
\end{cor}
\begin{proof}
This follows from the work of Brown \cite{brownmixedtate} and Racinet \cite{racinet}, as any element $\psi\in\gmot$ satisfies $\Delta(\psi)=0$.
\end{proof}

In low degree, we can translate this to a statement about elements of $\bg$ considered as polynomials.

\begin{example}\label{singleshuffexplicit}
For $f(x_1,x_2)\in\bg$ and $g(x_1,x_2,x_3)\in\bg$, we have
\begin{equation*}
\begin{split}
x\frac{\partial f}{\partial x_1}(0,x) &= f(x,-x),\\
\\
yz\left(\frac{\partial g}{\partial x_1}(0,y,z) \right. &- \left. \frac{\partial g}{\partial x_1}(0,y,-z)\right)=\\
y\left(g(y,z,-z)+g(-y,z,-z)\right)&+z\left(g(-y,y,-z)-g(-y,y,z)\right),\\
\\
yz\left(\frac{\partial g}{\partial x_1}(0,y,z) + \frac{\partial g}{\partial x_1}(0,y,-z)\right. &+\left.\frac{\partial g}{\partial x_2}(y,0,z) + \frac{\partial g}{\partial x_2}(y,0,-z)\right)=\\
y\left(g(y,z,-z)-g(-y,z,-z)\right)&-z\left(g(-y,y,-z)+g(-y,y,z)\right).
\end{split}
\end{equation*}
\end{example}

In order to better describe elements of $\bg$, we use the following lemma to transform our polynomial representation.

\begin{lem}\label{reducable}
For $f(x_1,x_2,\ldots,x_n)\in\bg$, we can write
$$f(x_1,\ldots,x_n)=x_1\ldots x_n(x_1-x_n)r(x_1,\ldots,x_n)$$
for some polynomial $r\in\Q[x_1,\ldots,x_n]$.
\end{lem}

\begin{proof}
We induct on the number of variables. For $n=2$, this follows from Theorem \ref{summedgen}. Now, suppose this factorisation holds for $f(x_1,x_2),\ g(x_1,\ldots,x_n)\in\bg$. We have
\begin{equation*}
\begin{split}
\{f,g\}= {} &-\sum_{i=1}^ng(x_1,\ldots,x_n)\left(\frac{1}{x_1}f(x_1,x_{n+1})-\frac{1}{x_n}f(x_n,x_{n+1})\right)\\
&+g(x_2,\ldots,x_{n+1}\left(\frac{1}{x_2}f(x_1,x_2)-\frac{1}{x_{n+1}}f(x_1,x_{n+1})\right)\\
-\frac{f(x_i,x_{i+1})}{x_i-x_{i+1}}&\left(\frac{1}{x_i}g(x_1,\ldots,x_i,x_{i+2}\ldots,x_{n+1})-\frac{1}{x_{i+1}}g(x_1,\ldots,x_{i-1},x_{i+1},\ldots,x_{n+1})\right)\\
\end{split}
\end{equation*}
Applying our induction hypothesis, we find
\begin{align*}
\{f,g\}= {} & x_1\ldots x_{n+1}r_f(x_1,x_2)\\
&\times\big((x_1-x_{n+1})r_g(x_1,x_3,\ldots,x_{n+1}) - (x_2-x_{n+1})r_g(x_2,\ldots,x_{n+1})\big)\\
&+\sum_{i=2}^{n-1} x_1\ldots x_{n+1}(x_1-x_{n+1})r_f(x_i,x_{i+1})\\
&\times\big(r_g(x_1,\ldots,x_i,x_{i+2}\ldots,x_{n+1})-r_g(x_1,\ldots,x_{i-1},x_{i+1},\ldots,x_{n+1})\big)\\
&+ x_1\ldots x_{n+1}r_f(x_n,x_{n+1})\\
&\times\big((x_1-x_{n})r_g(x_1,\ldots,x_{n})-(x_1-x_{n+1})r_g(x_1,\ldots,x_{n-1},x_{n+1})\big)\\
&-x_1\ldots x_{n+1}r_g(x_1,\ldots,x_n)\\
&\times\big((x_1-x_{n+1})r_f(x_1,x_{n+1})-(x_n-x_{n+1})r_f(x_n,x_{n+1})\big)\\
&-x_1\ldots x_{n+1}r_g(x_2,\ldots,x_{n+1})\\
&\times\big((x_1-x_2)r_f(x_1,x_2)-(x_1-x_{n+1})r_f(x_1,x_{n+1})\big).
\end{align*}

Considering only the terms not immediately divisible by $x_1\ldots x_{n+1}(x_1-x_{n+1})$, we reduce the problem to showing that
\begin{equation*}
\begin{split}
& -x_1\ldots x_{n+1}(x_2-x_{n+1})r_f(x_1,x_2) r_g(x_2,\ldots,x_{n+1})\\
&+ x_1\ldots x_{n+1}(x_1-x_n)r_f(x_n,x_{n+1})r_g(x_1,\ldots,x_n)\\
&+x_1\ldots x_{n+1}(x_n-x_{n+1})r_f(x_n,x_{n+1})r_g(x_1,\ldots,x_n)\\
&-x_1\ldots x_{n+1}(x_1-x_2)r_f(x_1,x_2)r_g(x_2,\ldots,x_{n+1})\\
= {} &-x_1\ldots x_{n+1}(x_1-x_{n+1})r_f(x_1,x_2) r_g(x_2,\ldots,x_{n+1})\\
&+ x_1\ldots x_{n+1}(x_1-x_{n+1})r_f(x_n,x_{n+1})r_g(x_1,\ldots,x_n)
\end{split}
\end{equation*}
is divisible by $x_1\ldots x_{n+1}(x_1-x_{n+1})$. This is clear, so we are done.

\end{proof}

\begin{defn}
For $f(x_1,\ldots,x_n)\in\bg$, define the reduced block polynomial to be 
$$r(x_1,\ldots,x_n):=\frac{f(x_1,\ldots,x_n)}{x_1\ldots x_n(x_1-x_n)}.$$
Define $\mathfrak{rbg}$ to be the bigraded $\Q$-vector space of reduced block polynomials.
\end{defn}

\begin{remark}
It may be useful to recall how the various degrees we assign to motivic iterated integrals relate to the reduced block polynomials. A reduced block polynomial $r(x_1,x_2,\ldots,x_n)$ of degree $N$ corresponds to elements of weight $N+n-1$ and block degree $n-1$.
\end{remark}

\section{The dihedral action}
As an immediate corollary to Proposition \ref{duality} we obtain:

\begin{lem}\label{reflection}
For all $r(x_1,\ldots,x_n)\in\mathfrak{rbg}$,
$$r(x_n,\ldots,x_1) = (-1)^nr(x_1,\ldots,x_n).$$
\end{lem}

\begin{defn}
We define a Lie algebra structure on $\mathfrak{rbg}$ via the Lie bracket

$$\{r_1,r_2\}(x_1,\ldots,x_{m+n-1}):=\frac{\{f_1,f_2\}(x_1,\ldots,x_{m+n-1})}{x_1\ldots x_{m+n-1}(x_1-x_{m+n-1})}$$
 for $r_1(x_1,\ldots,x_m)=\frac{f_1(x_1,\ldots,x_m)}{x_1\ldots x_m(x_1-x_m)}, r_2(x_1,\ldots,x_n)=\frac{f_2(x_1,\ldots,x_n)}{x_1\ldots x_n(x_1-x_n)}\in\mathfrak{rbg}$. We call this the reduced Ihara bracket. It produces a polynomial of degree $\deg(r_1)+\deg(r_2)$.
\end{defn}

We can explicitly compute this, and in the case of $r_1=r_1(x_1,x_2)$, we obtain a particularly nice formula.

\begin{prop}\label{reducedihara}
For $r(x_1,x_2),\ q(x_1,\ldots,x_{n-1})\in\mathfrak{rbg}$, the reduced Ihara bracket is given by
\begin{equation*}
\begin{split}
\{r,q\}(x_1,\ldots,x_n)=& \\
\sum_{i=1}^n r(x_i,x_{i+1})(q(x_1,\ldots,x_i,x_{i+2},\ldots,x_n)&-q(x_1,\ldots,x_{i-1},x_{i+1},\ldots,x_n))
\end{split}
\end{equation*}
where we consider indices modulo $n$.
\end{prop}

\begin{cor}\label{cyclicinvariance}
Elements $r(x_1,\ldots,x_n)\in\mathfrak{rbg}$ are invariant under cyclic rotations:
$$r(x_1,\ldots,x_n)=r(x_2,\ldots,x_n,x_1).$$
\end{cor}
\begin{proof}
This follows from a simple induction argument, using Lemma \ref{reflection} as our base case, and the natural cyclic symmetry in Proposition \ref{reducedihara}.
\end{proof}

\begin{remark}
Corollary \ref{cyclic} follows as an immediate corollary to this invariance.
\end{remark}

With this cyclic invariance, we can write down the general case of the reduced Ihara bracket quite succinctly.
\begin{cor}\label{generalreducedihara}
For $r(x_1,\ldots,x_m),q(x_1,\ldots,x_n)\in\rbg$, the reduced Ihara bracket is given by
\begin{equation*}
\begin{split}
\{r,q\}(x_1,\ldots,x_{m+n-1})&=\\
\sum_{i=1}^{m+n-1}r(x_i,\ldots,x_{i+m-1})&\left(q(x_{i+m},\ldots,x_{m+n-1},x_1,\ldots,x_i)\right.\\
&-\left. q(x_{i+m-1},\ldots,x_{m+n-1},x_1,\ldots,x_{i-1})\right)
\end{split}
\end{equation*}
where the indices are considered modulo $m+n-1$.
\end{cor}

Thus, we have an action of the dihedral group on $\rbg$, restricting to either the trivial or sign representation on the block-graded parts. 

\section{A differential relation}

We additionally obtain a differential relation, generalising the differential relation defining the generators of $\bg$.

\begin{defn}\label{differentialdefinition}
For $n\geq 2$, define the differential operator
$$\D_n:\Q[x_1,\ldots,x_n] \to \Q[x_1,\ldots,x_n]$$ by
$$\D_n := \prod_{i_1,\ldots,i_{n-1}\in\{0,1\}}\left(\partialx{1}+(-1)^{i_1}\partialx{2}+\cdots +(-1)^{i_{n-1}}\partialx{n}\right).$$
\end{defn}

\begin{thm}\label{differential}
For all $r(x_1,\ldots,x_n)\in\mathfrak{rbg}$,
$$\D_n r(x_1,\ldots,x_n)=0.$$
\end{thm}

\begin{proof}
We induct on $n$. For $n=2$, this follows from Corollary \ref{newgendef}. Suppose this holds for $q(x_1,\ldots,x_{n})\in\mathfrak{rbg}$.

Next define 
$$I_n:=\left\{M\in M_n(\mu_2) \ \mid\ M_{i,i}=1,\ \frac{M_{i+1,j}}{M_{i,j}}=\frac{M_{i+1,j+1}}{M_{i,j+1}}\right\},$$
and 
$$L_M:= \sum_{i=1}^n M_{1,i}\partialx{i} = \pm \sum_{i=1}^n M_{j,i}\partialx{i}.$$

Note that $\D_n = \prod_{M\in I_n}L_M$, and thus we have, for $r(x_1,x_2), q(x_1,\ldots,x_{n})\in\mathfrak{rbg}$,
\begin{equation*}
\begin{split}
\D_{n+1}\{r,q\}&(x_1,\ldots,x_{n+1}) =\\
 &\sum_{i=1}^n \D_{n+1}(r(x_i,x_{i+1})q(x_i,x_{i+2},\ldots,x_{i+n})-r(x_i,x_{i+1})q(x_{i+1},x_{i+2},\ldots,x_{i+n})\\
={}&\sum_{i=1}^n\sum_{S\subset I_{n+1}}\left(\prod_{M\in S}L_M \right)r(x_i,x_{i+1})\left(\prod_{M\in I_{n+1}\setminus S}L_M\right)q(x_i,x_{i+2},\ldots,x_{i+n})\\
&-\sum_{i=1}^n\sum_{S\subset I_{n+1}}\left(\prod_{M\in S}L_M\right) r(x_i,x_{i+1})\left(\prod_{M\in I_{n+1}\setminus S}L_M\right)q(x_{i+1},x_{i+2},\ldots,x_{i+n})
\end{split}
\end{equation*}
where we have used the cyclic invariance of $\mathfrak{rbg}$ and considering indices modulo $n+1$.

Next denote by $M[i_1,\ldots,i_k]$ the submatrix of $M$ obtained by restricting to rows and columns $i_1,\ldots,i_k$. We see that $L_M f(x_{i_1},\ldots,x_{i_k}) = L_{M[i_1,\ldots,i_k]}f(x_{i_1},\ldots,x_{i_k})$. 

Now, if $\{M[i,i+1]\ \mid\ M\in S\}=I_2$, then $(\prod_{M\in S}L_M )r(x_i,x_{i+1})=0$. Otherwise, we must have $M[i,i+1]=\bigl( \begin{smallmatrix}1 & 1\\ 1 & 1\end{smallmatrix}\bigr)$ for all $M\in S$, or $M[i,i+1]=\bigl( \begin{smallmatrix}1 & -1\\ -1 & 1\end{smallmatrix}\bigr)$ for all $M\in S$. In the first case, we must have all $M\in I_{n+1}$ with $M[i,i+1]=\bigl( \begin{smallmatrix}1 & -1\\ -1 & 1\end{smallmatrix}\bigr)$ contained in $I_{n+1}\setminus S$. The second case is similar. In either case, this implies that
$$\{M[i,i+2,\ldots,i+n]\ \mid\ M\in I_{n+1}\setminus S\}=\{M[i+1,\ldots,i+n]\ \mid\ I_{n+1}\setminus S\}=I_n$$
and so 
$$\left(\prod_{M\in I_{n+1}\setminus S}L_M\right)q(x_{i},x_{i+2},\ldots,x_{i+n})=\left(\prod_{M\in I_{n+1}\setminus S}L_M\right)q(x_{i+1},x_{i+2},\ldots,x_{i+n})=0$$
Thus $D_{n+1}\{r,q\}$=0.
\end{proof}

\begin{remark}
Note that, in sufficiently high degree, $r(x_1,\ldots,x_n)\in\ker \D_n$ is equivalent to $r(x_1,\ldots,x_n)\in\sum_{M\in I_n}\ker L_M$. This second condition clearly holds for $n=2$, and can easily be shown to be preserved by the Ihara bracket. Hence, we can equivalently state Theorem \ref{differential} as the following:
$$r(x_1,\ldots,x_n)\in\sum_{M\in I_n}\ker L_M\text{ for all }r(x_1,\ldots,x_n)\in\rbg.$$
\end{remark}

\begin{remark}\label{wildmassguessing}
We have shown that, in block degree 1, $\bg$ is isomorphic as a vector space to the bigraded vector space of homogeneous polynomials satisfying Theorem \ref{blockshuffle}, Example \ref{singleshuffexplicit}, Lemma \ref{reducable}, and whose reduced forms satisfy Corollary \ref{cyclicinvariance} and Theorem \ref{differential}. Note also that, as all these properties are preserved by the Ihara bracket, $\bg$ is a Lie subalgebra of the Lie algebra of homogeneous polynomials satisfying these properties. However, in block degree $b$, and weight $w$, we can only show that the dimension of the bigraded piece of the vector space of homogeneous polynomials satisfying these constraints is bounded above by $Cw^{b-1}$ for some constant $C$.
\end{remark}

\section{Deriving the Ihara action formula}
For elements of the double shuffle Lie algebra, the (linearised) Ihara action is given by the following \cite{zeta3}:

\begin{prop}
For $\sigma\in\Lie[e_0,e_1]$, $u\in\{e_0,e_1\}^\times$, the linearised Ihara action is given recursively by
\begin{equation}\label{ihararecursive}
\sigma\circ e_0^ne_1u := e_0^n\sigma e_1u - e_0^ne_1 \sigma^* u + e_0^n e_1(\sigma\circ u)
\end{equation}
where $(a_1\cdots a_n)^*:=(-1)^na_n\cdots a_1$.
\end{prop}

Translating the linearised Ihara action into the language of commutative variables, we find the following.

\begin{thm}\label{uglyIhara}
Let $f(x_1,\ldots,x_m)$ be the image of the block degree $m-1$ part of $\sigma\in\Lie[e_0,e_1]$, and $g\in\Q[x_1,\ldots,x_n]$. Then the linearised Ihara action is given by
\begin{equation*}
\begin{split}
(f\circ g)&(x_1,\ldots,x_{m+n-1}) =\\
&\sum_{i=1}^{n}(-1)^{(m+1)(i-1)}\frac{f(x_i,x_{i+1},\ldots,x_{i+m-1})}{x_i^2-x_{i+m-1}^2}\\
&\times\left(\left(1 +(-1)^{m+1} \frac{x_{i+m-1}}{x_i}\right)g(\bar x_1,\ldots, \bar x_i,x_{i+m},\ldots,x_{m+n-1})\right.\\
&\quad \left.-\left(1+(-1)^{m+1}\frac{x_i}{x_{i+m-1}}\right)g(\bar x_1,\ldots, \bar x_{i-1},x_{i+m-1},\ldots,x_{m+n-1})\right)
\end{split}
\end{equation*}
where we define $\bar x_i := (-1)^{m+1}x_i$.
\end{thm}
\begin{proof}
We start by writing, for $u=u_1\cdots u_n\in\{e_0,e_1\}^\times$,
$$\sigma\circ u_1\cdots u_n = \epsilon_0 \sigma u_1\cdots u_n + \sum_{i=1}^n \epsilon_i u_1\cdots u_i \sigma u_{i+1}\cdots u_n$$
where $\epsilon_i\in\{0,\pm 1\}$ for each $i$. We first claim that $\epsilon_i=0$ if $u_i=u_{i+1}$. We take here $u_0=e_0$ and $u_{n+1}=e_1$.

If $u_i=u_{i+1}=e_0$, then our recursive formula (\ref{ihararecursive}) shows $\epsilon_i=0$, as $\sigma$ does not `insert' between adjacent $e_0$. If $u_i=u_{i+1}=e_1$, then our recursion gives us terms of the form
\begin{equation*}
\begin{split}
\cdots &+ u_1\cdots u_{i-1} e_1 \sigma^* e_{1} u_{i+2}\cdots u_{n} + u_1\cdots u_{i-1}e_1 \sigma e_1 u_{i+2}\cdots u_n\\
&+ u_1\cdots u_{i-1} e_1e_1 \sigma^* u_{i+2}\cdots u_n + \cdots
\end{split}
\end{equation*}
As, for $\sigma\in\Lie[e_0,e_1]$, $\sigma+\sigma^*=0$, the terms corresponding to $u_1\cdots u_i \sigma  u_{i+1}\cdots u_n$ cancel, giving us that $\epsilon_i=0$. Hence, our block-polynomial formula will consist of a sum over the blocks of $u$, each corresponding to the insertion of $\sigma$ into a single block.

We will induct on the number of blocks in $u$. If $u$ consists of a single block, $u=(e_1e_0)^k$, and 
\begin{equation*}
\begin{split}
\sigma\circ u &= \sigma (e_1e_0)^k + e_1\sigma^* e_0 (e_1e_0)^{k-1} + e_1e_0\sigma (e_1e_0)^{k-2} +e_1e_0e_1\sigma^*e_0(e_1e_0)^{k-3}+\cdots\\
&=\sum_{i=0}^k [(e_1e_0)^i\sigma (e_1e_0)^{k-i} + (e_1e_0)^i e_1 \sigma^* e_0 (e_1e_0)^{k-1-i}].
\end{split}
\end{equation*}

Letting $f(x_1,\ldots,x_m)$ be the polynomial representing the block degree $n$ part of $\sigma$ and $g(x_1)=x_1^{2k+2}$ be the polynomial representing $u$, this is equivalent to the statement that
\begin{equation*}
\begin{split}
(f\circ g)&(x_1,\ldots,x_m)= \\
&\sum_{i=0}^k \left(\frac{x_1}{x_m}\right)^{2i}\frac{f(x_1,\ldots,x_m)g(x_m)}{x_m^2} + (-1)^{m+1}\frac{x_1}{x_m}\sum_{i=0}^{k-1}\left(\frac{x_1}{x_m}\right)^{2i}\frac{f(x_1,\ldots,x_m)g(m)}{x_m^2}\\
={}&\frac{f(x_1,\ldots,x_m)}{x_1^2-x_m^2}\left(\left(\frac{x_1}{x_m}\right)^{2k+2} - 1 +(-1)^{m+1} \left(\frac{x_1}{x_m}\right)^{2k+1} - (-1)^{m+1}\frac{x_1}{x_m}\right)g(x_m)\\
={}&\frac{f(x_1,\ldots,x_m)}{x_1^2-x_m^2}\left(g(x_1)-g(x_m)+(-1)^{m+1}\frac{x_m}{x_1}g(x_1)-(-1)^{m+1}\frac{x_1}{x_m}g(x_m)\right),
\end{split}
\end{equation*}
which is precisely the result given by the formula.

Now suppose our formula is correct for words consisting of $n-1$ blocks, and let $e_0ue_1$ be a word consisting of $n$ blocks, i.e. $e_0ue_1=b_1\cdots b_n$, represented by the monomial $g(x_1,\ldots,x_n)$. As we have merely appended a block onto the end of a word, the first $n-2$ terms of $(f\circ g)$ will be given by our formula, by our induction hypothesis. To see this, consider $(f\circ g_{\mathrm{alt}})$, where $g_{\mathrm{alt}}$ is the polynomial  corresponding to the word  $e_0u_{\mathrm{alt}}e_1=b_1\cdots b_{n-1}^\prime$. Here $b_{n-1}^\prime$ is the smallest block extending $b_{n-1}$ and ending on $e_1$. The Ihara action of any $\sigma\in\Lie[e_0,e_1]$ on $u$ and $u_{\mathrm{alt}}$ will produce terms that are identical upon swapping $b_{n-1}b_n \leftrightarrow b_{n-1}^\prime$ up to those terms in which $\sigma$ inserts into $b_{n-1}b_n$. Indeed, they will agree under this swapping until we consider terms in which $\sigma$ inserts beyond the end of $b_{n-1}$. Thus, it suffices to show that the formula holds for a word $e_0ue_1=b_1b_2$ of block degree 1.

We have 2 cases: the repeated letter in $e_0ue_1$ is $e_0$, or it is $e_1$.
In the first case, $u=(e_1e_0)^ke_0(e_1e_0)^l$ and
\begin{equation*}
\begin{split}
\sigma\circ u &=  \sum_{i=0}^k (e_1e_0)^i\sigma (e_1e_0)^{k-i}e_0(e_1e_0)^l + \sum_{i=0}^{k-1}(e_1e_0)^ie_1\sigma^* e_0(e_1e_0)^{k-1-i}e_0(e_1e_0)^l\\
&+\sum_{i=0}^l(e_1e_0)^ke_0(e_1e_0)^i\sigma (e_1e_0)^{l-i} + \sum_{i=0}^{l-1}(e_1e_0)^ke_0(e_1e_0)^ie_1\sigma^*e_0(e_1e_0)^{l-1-i}
\end{split}
\end{equation*}
In terms of commutative polynomials, after summing the geometric series, we obtain
\begin{equation*}
\begin{split}
(f\circ g)(x_1,\ldots,x_{m+1}) ={}&\frac{f(x_1,\ldots,x_m)}{x_1^2-x_m^2}\left(\left(\frac{x_1}{x_m}\right)^{2k} - 1 \right.\\
&\left.+ (-1)^{m+1}\left(\frac{x_1}{x_m}\right)^{2k+1}-(-1)^{m+1}\frac{x_1}{x_m}\right)g(x_m,x_{m+1})\\
&+\frac{f(x_2,\ldots,x_{m+1})}{x_2^2-x_{m+1}^2}\left(\left(\frac{x_2}{x_{m+1}}\right)^{2l+2} - 1\right.\\
&\left.+ (-1)^{m+1}\left(\frac{x_2}{x_{m+1}}\right)^{2l+1}-(-1)^{m+1}\frac{x_2}{x_{m+1}}\right)g(x_1,x_{m+1}).
\end{split}
\end{equation*}
Simplifying, and noting that $g(x_1,x_2)= x_1^{2k+1}x_2^{2l+2}$, we obtain
\begin{equation*}
\begin{split}
(f\circ g)(x_1,\ldots,x_{m+1})&= \frac{f(x_1,\ldots,x_m)}{x_1^2-x_m^2}\left(\frac{x_m}{x_1}g(x_1,x_m+1) -g(x_m,x_m+1)\right.\\
&\left.+(-1)^{m+1}g(x_1,x_{m+1})-(-1)^{m+1}\frac{x_1}{x_m}g(x_m,x_{m+1})\right)\\
&+\frac{f(x_2,\ldots,x_{m+1})}{x_2^2-x_{m+1}^2}\bigg(g(x_1,x_2)-g(x_1,x_{m+1})\\
&\left.+(-1)^{m+1}\frac{x_{m+1}}{x_2}g(x_1,x_2) - (-1)^{m+1}\frac{x_2}{x_{m+1}}g(x_1,x_{m+1})\right).
\end{split}
\end{equation*}
Considering parity, and defining $\bar x_i:=(-1)^{m+1}x_i$, we can rewrite this as
\begin{equation*}
\begin{split}
(f\circ g)(x_1,\ldots,x_{m+1})=& (-1)^{(0)(m+1)}\frac{f(x_1,\ldots,x_m)}{x_1^2-x_m^2}\\
&\times\left(g(\bar x_1,x_{m+1})+(-1)^{m+1}\frac{x_m}{x_1}g(\bar x_1,x_{m+1})\right.\\
&\left. - g(x_m,x_{m+1}) + (-1)^{m+1}\frac{x_m}{x_{m+1}}g(x_m,x_{m+1})\right)\\
&+(-1)^{m+1}\frac{f(x_2,\ldots,x_{m+1})}{x_2^2-x_{m+1}^2}\\
&\times\left(g(\bar x_1,\bar x_2) - (-1)^{m+1}\frac{x_{m+1}}{x_2}g(\bar x_1,\bar x_2)\right.\\
&\left. -g(\bar x_1,x_{m+1})+(-1)^{m+1}\frac{x_2}{x_{m+1}}g(\bar x_1,x_{m+1})\right)
\end{split}
\end{equation*} giving the desired formula. The second case follows similarly.

Hence, our general formula is $(f\circ g)(x_1,\ldots,x_{m+n-1}) =A\pm B\pm C$, where
\begin{equation*}
\begin{split}
A ={}& \sum_{i=1}^{n-2}(-1)^{(m+1)(i-1)}\frac{f(x_i,x_{i+1},\ldots,x_{i+m-1})}{x_i^2-x_{i+m-1}^2}\\
&\times\left(\left(1 +(-1)^{m+1} \frac{x_{i+m-1}}{x_i}\right)g(\bar x_1,\ldots, \bar x_i,x_{i+m},\ldots,x_{m+n-1})\right.\\
&\left.-\left(1+(-1)^{m+1}\frac{x_i}{x_{i+m-1}}\right)g(\bar x_1,\ldots, \bar x_{i-1},x_{i+m-1},\ldots,x_{m+n-1})\right),\\
B ={}&(-1)^{(m+1)(n-2)}\frac{f(x_{n-1},\ldots,x_{m+n-2})}{x_{n-1}^2-x_{n+m-2}^2}\\
&\times\left(\left(1 +(-1)^{m+1} \frac{x_{n+m-2}}{x_{n-1}}\right)g(\bar x_1,\ldots, \bar x_{n-1},x_{m+n-1})\right.\\
&\left.-\left(1+(-1)^{m+1}\frac{x_{n-1}}{x_{n+m-2}}\right)g(\bar x_1,\ldots, \bar x_{n-2},x_{n+m-2},x_{m+n-1})\right),\\
C ={}&(-1)^{(m+1)(n-1)}\frac{f(x_{n},\ldots,x_{m+n-1})}{x_{n}^2-x_{m+n-1}^"}\\
&\times\left(\left(1 +(-1)^{m+1} \frac{x_{n+m-1}}{x_{n}}\right)g(\bar x_1,\ldots, \bar x_{n})\right.\\
&\left.-\left(1+(-1)^{m+1}\frac{x_{n}}{x_{n+m-1}}\right)g(\bar x_1,\ldots, \bar x_{n-1},x_{m+n-1})\right),\\
\end{split}
\end{equation*}
 and the signs of $B$ and $C$ agree. To fix this sign, we need only to consider the sign of the term corresponding to $(\frac{x_{n-1}}{x_{m+n-2}})^2\frac{f(x_{n-1},\ldots,x_{m+n-2})g(x_1,\ldots,x_{n-2},x_{m+n-2},x_{m+n-1})}{x_{m+n-2}^2}$. This corresponds to inserting $\sigma$ after the first two letters of the $(n-1)$\textsuperscript{th} block. The sign will be positive if this block starts with an $e_1$ and must have the same sign as $(-1)^{m+1}$ otherwise. Let $g(x_1,\ldots,x_n)=x_1^{d_1}\ldots x_n^{d_n}$. Then by Lemma \ref{blockparity} the $(n-1)$\textsuperscript{th} block starts with $e_0$ if $d_1+\cdots+d_{n-2}\equiv n-2 \mod2$, and $e_1$ otherwise. Thus, the sign of the term corresponding to $(\frac{x_{n-1}}{x_{m+n-2}})^2\frac{f(x_{n-1},\ldots,x_{m+n-2})g(x_1,\ldots,x_{n-2},x_{m+n-2},x_{m+n-1})}{x_{m+n-2}^2}$ is $(-1)^{(m-1)(1 + d_1+\cdots + d_{n-2} - n +2)}$. Comparing this with our formula, we see that the final two terms must appear with a positive sign, giving the desired result.

\end{proof}

To obtain (\ref{prettyIhara}), we must translate this across into the `depth-signed' convention. Specifically, we must find the action of the map $e_1\mapsto -e_1$ in terms of commutative variables.

\begin{lem}\label{depthsigned}
The automorphism $\qpoly\to\qpoly$ given by $e_1\mapsto -e_1$, is equivalent under the isomorphism (\ref{isomorph}) to the map
\begin{equation}\label{depthsigning}
\begin{split}
\Q[x_1,\ldots,x_n]&\to\Q[x_1,\ldots,x_n]\\
f(x_1,\ldots,x_n)&\mapsto(-1)^{\left\lceil\frac{l}{2}\right\rceil}f(-x_1,x_2,\ldots,(-1)^nx_n)
\end{split}
\end{equation}
for $f$ a homogeneous polynomial of degree $l+2$.
\end{lem}

\begin{proof}
Note that it suffices to show that, for a word $w$ of length $l$ and depth $d$, with $\pi_{\text{bl}}(w)=x_1^{d_1}\ldots x_n^{d_n}$, this congruence holds

$$d\equiv \left\lceil\frac{l}{2}\right\rceil + d_1 + d_3+\cdots\mod2.$$

We will induct on the number of blocks in $e_0we_1$. If $e_0we_1$ consists of a single block, then $e_0we_1 = (e_0e_1)^{\frac{l}{2}+1}$, and so $d= \frac{l}{2}$, and $d_1=l+2$. Thus the result holds.

Suppose the result holds for $w$ such that $e_0we_1$ consists of $n$ blocks. Let $e_0we_1=w^\prime w_{d_{n+1}}$ be a word of length $l+2$ and depth $d+1$,consisting of $n+1$ blocks, where $w_{d_{n+1}}$ is a single block of length $d_{n+1}$ and $w^\prime$ is a word of length $l^\prime$ and depth $d^\prime$. Suppose $\pi_{\text{bl}}(w)=x_1^{d_1}\ldots x_n^{d_n}x_{n+1}^{d_{n+1}}$.

If $l^\prime$ is even, then $w^\prime=e_0 u e_1$ consists of $n\equiv 1\mod2$ blocks, and $d_{n+1}$ must be odd. So, by induction,
$$d^\prime - 1 \equiv \frac{l^\prime-2}{2}+\sum_{1\leq 2i+1\leq n}d_{2i+1}\mod2.$$
Thus 
\begin{equation*}
\begin{split}
d &= d^\prime + \left\lceil\frac{d_{n+1}}{2}\right\rceil - 1\\
&\equiv  \frac{l^\prime-2}{2}+\sum_{1\leq 2i+1\leq n}d_{2i+1}+ \left\lceil\frac{d_{n+1}}{2}\right\rceil\mod2\\
&\equiv \left\lceil\frac{l^\prime + d_{n+1}-2}{2}\right\rceil +\sum_{1\leq 2i+1\leq n+1}d_{2i+1}\mod2\\
&\equiv \left\lceil\frac{l}{2}\right\rceil +\sum_{1\leq 2i+1\leq n+1}d_{2i+1}\mod2,
\end{split}
\end{equation*}
 and so the result holds. Similar considerations for $l^\prime$ odd prove the result in general.
\end{proof}

Applying this transformation, and simplifying, we obtain Proposition \ref{Iharanice}, giving the formula
\begin{equation*}
\begin{split}
(f\circ g)(x_1,\ldots,x_{m+n-1}) &= \sum_{i=1}^n \frac{f(x_i,\ldots,x_{i+m-1})}{x_i-x_{i+m-1}}\left (\frac{1}{x_i}g(x_1,\ldots,x_i,x_{i+m},\ldots,x_{m+n-1})\right.\\
&\left.\hspace{6em} -\frac{1}{x_{i+m-1}}g(x_1,\ldots,x_{i-1},x_{i+m-1},\ldots,x_{m+n-1})\right).
\end{split}
\end{equation*}

\section{On a conjecture of Charlton}
In \cite{charthesis}, Charlton proposes the following conjecture as a generalisation of one proposed in \cite{???}, and has verified them numerically in many cases.

\begin{conj}[Conjecture 8.2 \cite{charthesis}]
For any block decomposition $(l_1,\ldots,l_{2n+1})$ of even weight $N$
$$\sum_{\sigma\in\Sym(1,3,\ldots,2n+1)}\sgn(\sigma)\im(l_{\sigma(1)},l_2,l_{\sigma(3)},\ldots,l_{\sigma(2n+1)}) = 0$$
where we sum over all permutations of the odd indices.
\end{conj}

Using the block-graded machinery, we can prove a weaker version of this conjecture, modulo products and terms of lower block degree.
\begin{prop}\label{charlalt}
For any block decomposition $(l_1,\ldots,l_{2n+1})$ of even weight $N$
$$\sum_{\sigma\in\Sym(1,3,\ldots,2n+1)}\sgn(\sigma)\ibl(l_{\sigma(1)},l_2,l_{\sigma(3)},\ldots,l_{\sigma(2n+1)}) = 0$$
modulo products, where we sum over all permutations of the odd indices.
\end{prop}
\begin{proof}
First note that this is equivalent to the statement that
$$\sum_{\sigma\in\Sym(1,3,\ldots,2n+1)}\sgn(\sigma)f(x_{\sigma(1)},x_2,x_{\sigma(3)},\ldots,x_{\sigma(2n+1)}) = 0$$
for any $f(x_1,\ldots,x_{2n+1})\in\bg$. Writing $$f(x_1,\ldots,x_{2n+1})=x_1\ldots x_{2n+1}(x_1-x_{2n+1})r(x_1,\ldots,x_{2n+1}),$$ with $r(x_1,\ldots,x_{2n+1})\in\rbg$, we find that it is sufficient to show that
\begin{align*}
&\sum_{\substack{\sigma\in\Sym(1,3,\ldots,2n+1)\\ \sigma(1)=1}}\sgn(\sigma)r(x_{\sigma(1)},x_2,x_{\sigma(3)},\ldots,x_{\sigma(2n+1)}) \\
={}&\sum_{\substack{\sigma\in\Sym(1,3,\ldots,2n+1)\\ \sigma(2n+1)=1}}\sgn(\sigma)r(x_{\sigma(1)},x_2,x_{\sigma(3)},\ldots,x_{\sigma(2n+1)}).\end{align*}
Denote by $\mathcal{S}_{i,n}:=\{\sigma\in\Sym(1,3,\ldots,2n+1)\mid \sigma(1)=i\}$. We then proceed by induction. The case $n=1$ is simple:
$$r(x_1,x_2,x_3)=-r(x_3,x_2,x_1),$$
which follows from Lemma \ref{reflection}. Now suppose the claim holds for $k=1,\ldots,n$. It then sufficies to show that it holds for $\{r,q\}$ for $r(x_1,x_2,x_3),\ q(x_1,\ldots,x_{2n+1})\in\rbg$. From Corollary \ref{generalreducedihara}, we can write
\begin{align*}
\{r,q\}(x_1,\ldots,x_{2n+3})={}&\sum_{i=1}^{2n+3}r(x_i,x_{i+1},x_{i+2})(q(x_{i+3},\ldots,x_{2n+3},x_1,\ldots,x_i)\\
&-q(x_{i+2},\ldots,x_{2n+3},x_1,\ldots,x_{i-1})\\
={}&\sum_{\sigma\in\Cyc_{2n+3}}\sigma\cdot F(x_1,\ldots,x_{2n+3}) -\sigma\cdot (1\ 3)\cdot F(x_1,\ldots,x_{2n+3})
\end{align*}
where $F(x_1,\ldots,x_{2n+3}):=r(x_1,x_2,x_3)q(x_1,x_4,\ldots,x_{2n+3})$, and $\sigma\cdot F(x_1,\ldots,x_{2n+3}):=F(x_{\sigma(1)},\ldots,x_{\sigma(2n+3)})$ for any permutation $\sigma$. The claim is then that
\begin{align*}
&\sum_{\sigma\in \mathcal{S}_{1,n+1}}\sum_{\tau\in\Cyc_{n+1}}\sgn(\sigma)\sigma\cdot\tau\cdot (F -(1\ 3)\cdot F)\\
={}&\sum_{\sigma\in \mathcal{S}_{2n+3,n+1}}\sum_{\tau\in\Cyc_{n+1}}\sgn(\sigma)\sigma\cdot\tau\cdot (F -(1\ 3)\cdot F)
\end{align*}

We divide the sums up into a number of sub-sums, according to where the permutations send $\{1,2,3\}$. Let $\mathcal{R}_{a,b,c}:=\{\sigma\in\text{S}_{2n+3}\mid \{\sigma(1),\sigma(2),\sigma(3)\}=\{a,b,c\}\}$. Note that $\sigma\in\mathcal{R}_{a,b,c}\Rightarrow \sigma\cdot (1\ 3)\in\mathcal{R}_{a,b,c}$. Hence
\begin{align*}
&\sum_{\sigma\in\mathcal{S}_{i,n+1}}\sum_{\substack{\tau\in\Cyc_{2n+3}\\ \sigma\cdot\tau\in\mathcal{R}_{a,b,c}}}\sgn(\sigma)\sigma\cdot\tau\cdot(F-(1\ 3)\cdot F)=\\
r(a,b,c)&\sum_{\sigma\in\mathcal{S}_{i,n+1}}\sum_{\substack{\tau\in\Cyc_{2n+3}\\ \sigma\cdot\tau\in\mathcal{R}_{a,b,c}}}\sgn(\sigma_{a,b,c})\sgn(\sigma)\sigma\cdot\tau\cdot(q(x_1,x_4,\ldots,x_{2n+3})-q(x_3,\ldots,x_{2n+3}))
\end{align*}
where $\sgn(\sigma_{a,b,c})$ is the sign of the permutation of $\{\sigma(1),\sigma(2),\sigma(3)\}$ taking the ordered tuple $(\sigma(1),\sigma(2),\sigma(3))$ to $(a,b,c)$.

First suppose two of $\{a,b,c\}$ are even. As $\sigma$ leaves even indices fixed, and $2n+3$ is odd, we must have that $\tau(i)=i+2k+1$ for some unique $1\leq 2k+1\leq 2n-1$ . Furthermore, we must have $a=c-2$ as the even indices. Thus, for $i=1,2n+3$, we must have
\begin{align*}
&\sum_{\sigma\in\mathcal{S}_{i,n+1}}\sum_{\substack{\tau\in\Cyc_{2n+3}\\ \sigma\cdot\tau\in\mathcal{R}_{a,b,c}}}\sgn(\sigma_{a,b,c})\sgn(\sigma)\sigma\cdot\tau\cdot\big(q(x_1,x_4,\ldots,x_{2n+3})-q(x_3,\ldots,x_{2n+3})\big)\\
={}&\sum_{\substack{\sigma\in\mathcal{S}_{i,n+1}\\ \sigma(2k+3)=b}}\sgn(\sigma)\sigma\cdot\big(q(x_{2k+2},x_{2k+5},\ldots,x_{2k+1})-q(x_{2k+4},x_{2k+5},\ldots,x_{2k+1})\big)\\
={}&\sum_{\substack{\sigma\in\mathcal{S}_{i,n+1}\\ \sigma(2k+3)=b}}\begin{aligned}[t]\sgn(\sigma)\sigma\cdot &\big(q(x_1,\ldots,x_{2k+2},x_{2k+5},\ldots,x_{2n+3})\\
&-q(x_1,\ldots,x_{2k+1},x_{2k+4},\ldots,x_{2n+3})\big)\end{aligned}\\
={}&N\sum_{\substack{\sigma\in\widetilde{\mathcal{S}}\\ \sigma(i)=1}}\begin{aligned}[t]\sgn(\sigma)\sigma\cdot &\big(q(x_1,\ldots,x_{2k+2},x_{2k+5},\ldots,x_{2n+3})\\
&-q(x_1,\ldots,x_{2k+1},x_{2k+4},\ldots,x_{2n+3})\big)\end{aligned}
\end{align*}
where $\widetilde{\mathcal{S}}=\Sym(1,3,\ldots,2k+1,2k+5,\ldots,2n+3)$ and $N$ is the index of $\{\sigma\in\mathcal{S}_{i,n+1}\mid \sigma(2k+3)=b\}$ in $\mathcal{S}_{i,n+1}$, which we note is independent of $i$. Thus, by the induction hypothesis

\begin{align*}
&\sum_{\sigma\in\mathcal{S}_{1,n+1}}\sum_{\substack{\tau\in\Cyc_{2n+3}\\ \sigma\cdot\tau\in\mathcal{R}_{a,b,c}}}\sgn(\sigma)\sigma\cdot\tau\cdot(F-(1\ 3)\cdot F)\\
={}&\sum_{\sigma\in\mathcal{S}_{2n+3,n+1}}\sum_{\substack{\tau\in\Cyc_{2n+3}\\ \sigma\cdot\tau\in\mathcal{R}_{a,b,c}}}\sgn(\sigma)\sigma\cdot\tau\cdot(F-(1\ 3)\cdot F)
\end{align*}

If two of $\{a,b,c\}$ are odd, then we must have that $\tau(2)$ is even, or $\tau(2)\in\{2n+3,1\}$. We will consider only the first cases, as the others follow similarly. Without loss of generality, we assume $b$ is even, and so $\tau(i)=i+b-2$. Thus
\begin{align*}
&\sum_{\sigma\in\mathcal{S}_{i,n+1}}\sum_{\substack{\tau\in\Cyc_{2n+3}\\ \sigma\cdot\tau\in\mathcal{R}_{a,b,c}}}\sgn(\sigma_{a,b,c})\sgn(\sigma)\sigma\cdot\tau\cdot\big(q(x_1,x_4,\ldots,x_{2n+3})-q(x_3,\ldots,x_{2n+3})\big)\\
={}&\sum_{\substack{\sigma\in\mathcal{S}_{i,n+1}\\ \sigma(b-1)=a, \sigma(b+1)=c}}\sgn(\sigma)\sigma\cdot\big(q(x_{b-1},x_{b+2},\ldots,x_{b-2})-q(x_{b+1},\ldots,x_{b-2})\big)\\
&-\sum_{\substack{\sigma\in\mathcal{S}_{i,n+1}\\ \sigma(b-1)=c, \sigma(b+1)=a}}\sgn(\sigma)\sigma\cdot\big(q(x_{b-1},x_{b+2},\ldots,x_{b-2})-q(x_{b+1},\ldots,x_{b-2})\big)\\
={}&\sum_{\substack{\sigma\in\mathcal{S}_{i,n+1}\\ \sigma(b-1)=a, \sigma(b+1)=c}}\sgn(\sigma)\sigma\cdot\big(q(x_{b-1},x_{b+2},\ldots,x_{b-2})-q(x_{b+1},\ldots,x_{b-2})\big)\\
&+\sum_{\substack{\sigma\in\mathcal{S}_{i,n+1}\\ \sigma(b-1)=a, \sigma(b+1)=c}}\sgn(\sigma)\sigma\cdot\big(q(x_{b+1},x_{b+2},\ldots,x_{b-2})-q(x_{b-1},\ldots,x_{b-2})\big)\\
={}&{}0
\end{align*}
Hence, the result follows.
\end{proof}

Similar arguments can also establish the block-graded, modulo products versions of the odd-weight conjectures.
\begin{thm}
For any choice of $l_1,\ldots,l_6$, the following hold.
\begin{align*}
\sum_{\sigma\in\Sym(1,3)}\sum_{\tau\in\Sym(2,4)}\sgn(\sigma)\sgn(\tau)\ibl(l_{\sigma(1)},l_{\tau(2)},l_{\sigma(3)},l_{\tau(4)}) = 0\\
\sum_{\sigma\in\Sym(1,3,5)}\sum_{\tau\in\Sym(2,4,6)}\sgn(\sigma)\sgn(\tau)\ibl(l_{\sigma(1)},l_{\tau(2)},l_{\sigma(3)},l_{\tau(4)},l_{\sigma(5)},l_{\tau(6)}) = 0
\end{align*}
\end{thm}

\section{Further Remarks}
As $\bg$ is a free Lie algebra, we have a non-canonical isomorphism $\gmot\cong \bg$, and hence, we should be able to lift relations such as cyclic insertion and the block shuffle relations to $\gmot$ . One would hope that we could follow the example of Brown in \cite{anatomy}, in which he lifts solutions from $\ls$ to $\dmr$, but it is as yet unclear how this should work. Progress on this has been made by Hirose and Sato, who propose an ungraded version of the block shuffle relations, but it is an area ripe for further consideration. If one were to extend the relations described in this paper to a complete set of relations among block graded motivic multiple zeta values, a successful lifting of these relations could provide a complete description of all relations among motivic multiple zeta values, and give an approach to tackling questions about the completeness of the associator and double shuffle relations.

 In block degree one, we have shown the relations discussed are complete, but in general, further relations are required. However, one can show that, if $f_b(N)$ is the dimension of the weight $N$ vector space cut out by the relations described here in block degree $b$, then
$(\dim \gr^\mathcal{B}_b \mathcal{H}_N)/f_b(N)$ tends to a positive finite limit as $N\to\infty$. Furthermore, numerical evidence suggests that the ratio is close to $1$ for all $N$ and small $b$. However, further relations are needed, even in block degree $2$. While the projections into highest block degree of the double shuffle relations arising from simple products $\zm(w_1)\zm(w_2)$ numerically seem to be consequences of the relations discussed here, the block graded parts of certain linear combinations are not. For example
$$\zm(3,2,3)+2\zm(2,3,3)+\frac{7}{2}(\zm(2,4,2)+2\zm(4,2,2))=0\text{ modulo products}$$
is a missing graded relation in weight 8. Nevertheless, this is a promising start and provides an interesting new lens through which to consider MZVs. The construction of genuine relations, and the connection between these block-graded relations and known relations is explored further in the author's doctoral thesis \cite{mythesis}.

\bibliographystyle{abbrv}
\bibliography{bibliography_MZV_BLock}

\end{document}